\documentclass[12pt,reqno]{amsart}
\usepackage{amsmath} 

\usepackage{amsthm}
\usepackage{amssymb}
\usepackage{graphics}
\usepackage{latexsym}
\usepackage{comment}

\numberwithin{equation}{section}
\newtheorem{thm}{Theorem}[section]
\newtheorem{prop}[thm]{Proposition}
\newtheorem{lemm}[thm]{Lemma}
\newtheorem{cor}[thm]{Corollary}

\theoremstyle{remark}
\newtheorem{rem}{Remark}[section]
\newtheorem{defn}{Definition}

\newcommand{\BBB}{\mathbb}
\newcommand{\R}{{\BBB R}}
\newcommand{\Z}{{\BBB Z}}

\newcommand{\N}{{\BBB N}}
\newcommand{\C}{{\BBB C}}

\newcommand{\ZZ}{\mathcal{Z}}
\newcommand{\FT}{\mathcal{F}}
\newcommand{\ee}{\mbox{\boldmath $1$}}

\newcommand{\al}{\alpha}
\newcommand{\be}{\beta}
\newcommand{\ga}{\gamma}

\newcommand{\p}{\partial}

\newcommand{\kuuhaku}{\text{}}

\newcommand{\F}{\mathcal{F}}

\title[WP and scattering for system of qDNLS]{
}

\author[H. Hirayama]{
}
\address[H. Hirayama]{
}
\email[H. Hirayama]{m08035f@math.nagoya-u.ac.jp}


\begin{document}
\begin{center}
{\bf \normalsize WELL-POSEDNESS  AND SCATTERING FOR A SYSTEM OF QUADRATIC DERIVATIVE
NONLINEAR SCHR\"ODINGER EQUATIONS WITH LOW REGULARITY INITIAL DATA}

\smallskip \bigskip {\sc Hiroyuki Hirayama}

\smallskip {\footnotesize Graduate School of Mathematics, Nagoya University\\
Chikusa-ku, Nagoya, 464-8602, Japan}
\end{center}
\vspace{-2ex}
\begin{abstract}
In the present paper, we consider the Cauchy problem of a system of quadratic derivative nonlinear 
Schr\"odinger equations which was introduced by M. Colin and T. Colin (2004) as a model of laser-plasma interaction. 
The local existence of the solution of the system in the Sobolev space $H^{s}$ for $s>d/2+3$ is proved by M. Colin and T. Colin.  
We prove the well-posedness of the system with low regularity initial data. 
For some cases, we also prove the well-posedness and the scattering at the scaling critical regularity  
by using $U^{2}$ space and $V^{2}$ space which are applied to prove 
the well-posedness and the scattering for KP-II equation at the scaling critical regularity by Hadac, Herr and Koch (2009).  
\end{abstract}
\maketitle
\setcounter{page}{001}


\section{Introduction\label{intro}}
We consider the Cauchy problem of the system of Schr\"odinger equations:
\begin{equation}\label{NLS_sys}
\begin{cases}
\displaystyle (i\p_{t}+\alpha \Delta )u=-(\nabla \cdot w )v,\hspace{2ex}(t,x)\in (0,\infty )\times \R^{d} \\
\displaystyle (i\p_{t}+\beta \Delta )v=-(\nabla \cdot \overline{w})u,\hspace{2ex}(t,x)\in (0,\infty )\times \R^{d} \\
\displaystyle (i\p_{t}+\gamma \Delta )w =\nabla (u\cdot \overline{v}),\hspace{2ex}(t,x)\in (0,\infty )\times \R^{d} \\
(u(0,x), v(0,x), w(0,x))=(u_{0}(x), v_{0}(x), w_{0}(x)),\hspace{2ex}x\in \R^{d}
\end{cases}
\end{equation}
where $\al$, $\be$, $\ga\in \R\backslash \{0\}$ and the unknown functions $u$, $v$, $w$ are $d$-dimensional complex vector valued. 
The system (\ref{NLS_sys}) was introduced by Colin and Colin in \cite{CC04} 
as a model of laser-plasma interaction.  
(\ref{NLS_sys}) is invariant under the following scaling transformation:
\begin{equation}\label{scaling_tr}
A_{\lambda}(t,x)=\lambda^{-1}A(\lambda^{-2}t,\lambda^{-1}x)\ \ (A=(u,v,w) ), 
\end{equation}
and the scaling critical regularity is $s_{c}=d/2-1$. 
The aim of this paper is to prove the well-posedness and the scattering of (\ref{NLS_sys}) 
in the scaling critical Sobolev space. 

First, we introduce some known results for related problems. 
The system (\ref{NLS_sys}) has quadratic nonlinear terms which contains a derivative. 
A derivative loss arising from the nonlinearity makes the problem difficult. 
In fact, Mizohata (\cite{Mi85}) proved that a necessary condition for the $L^{2}$ well-posedness of the problem:
\[
\begin{cases}
i\partial_{t}u-\Delta u=b_{1}(x)\nabla u,\ t\in \R ,\ x\in \R^{d},\\
u(0,x)=u_{0}(x),\ x\in \R^{d}
\end{cases}
\]
is the uniform bound
\[
\sup_{x\in \R^{n},\omega \in S^{n-1},R>0}\left| {\rm Re}\int_{0}^{R}b_{1}(x+r\omega )\cdot \omega dr\right| <\infty.
\]
Furthermore, Christ (\cite{Ch}) proved that the flow map of the Cauchy problem:
\begin{equation}\label{1dqdnls}
\begin{cases}
i\partial_{t}u-\partial_{x}^{2}u=u\partial_{x}u,\ t\in \R,\ x\in \R,\\
u(0,x)=u_{0}(x),\ x\in \R
\end{cases}
\end{equation}
is not continuous on $H^{s}$ for any $s\in \R$. 
While, there are positive results for the Cauchy problem:
\begin{equation}\label{ddqdnls}
\begin{cases}
i\partial_{t}u-\Delta u=\overline{u}(\nabla \cdot \overline{u}),\ t\in \R,\ x\in \R^{d},\\
u(0,x)=u_{0}(x),\ x\in \R^{d} .
\end{cases}
\end{equation}
Gr\"unrock (\cite{Gr}) proved that (\ref{ddqdnls}) is globally well-posed in $L^{2}$ for $d=1$ and 
locally well-posed in $H^{s}$ for $d\geq 2$ and $s>s_{c}$ ($=d/2-1$). 
For more general problem:
\begin{equation}\label{genqdnls}
\begin{cases}
i\partial_{t}u-\Delta u=P(u,\overline{u},\nabla u,\nabla \overline{u}),\ t\in \R,\ x\in \R^{d},\\
u(0,x)=u_{0}(x),\ x\in \R^{d},\\
P\ {\rm is\ a\ polynomial\ which\ has\ no\ constant\ and\ linear\ terms},
\end{cases}
\end{equation}
there are many positive results for the well-posedness 
in the weighted Sobolev space (\cite{Be06}, \cite{Be08}, \cite{Chi95}, \cite{Chi99}, \cite{KPV98}, \cite{St07}). 
Kenig, Ponce and Vega (\cite{KPV98}) also obtained that (\ref{genqdnls}) 
is locally well-posed in $H^{s}$ (without weight)  for large enough $s$ when $P$ has no quadratic terms. 

The Benjamin--Ono equation:
\begin{equation}\label{BOeq}
\partial_{t}u+H\partial_{x}^{2}u=u\partial_{x}u,\ (t,x)\in \R \times \R
\end{equation}
is also related to the quadratic derivative nonlinear Schr\"odinger equation. 
It is known that the flow map of (\ref{BOeq}) is not uniformly continuous 
on $H^{s}$ for $s>0$ (\cite{KT05}). 
But the Benjamin--Ono equation has better structure than the equation (\ref{1dqdnls}). 
Actually,  
Tao (\cite{Ta04}) proved that (\ref{BOeq}) is globally well-posed in $H^{1}$ 
by using the gauge transform. 
Furthermore, Ionescu and Kenig (\cite{IK07}) proved that (\ref{BOeq}) is globally well-posed in 
$H_{r}^{s}$ for $s\geq 0$, where $H_{r}^{s}$ is the Banach space of the all real valued function $f\in H^{s}$. 

Next, we introduce some known results for  systems of quadratic nonlinear derivative Schr\"odinger equations. 
Ikeda, Katayama and Sunagawa (\cite{IKS13}) considered (\ref{NLS_sys}) with null form nonlinearity and obtained the small data global existence and the scattering 
in the weighted Sobolev space for the dimension $d\geq 2$ under the condition $\alpha \beta \gamma (1/\alpha -1/\beta -1/\gamma )=0$. 
While, Ozawa and Sunagawa (\cite{OS13}) gave the examples of the quadratic derivative nonlinearity which causes the small data blow up for a system of Schr\"odinger equations. 
As the known result for (\ref{NLS_sys}), we introduce the work by Colin and Colin (\cite{CC04}). 
They proved that the local existence of the solution of (\ref{NLS_sys}) for $s>d/2+3$. 
There are also some known results for a system of Schr\"odinger equations with no derivative nonlinearity 
(\cite{CCO09_1}, \cite{CCO09_2}, \cite{CO12}, \cite{HLN11}, \cite{HLO11}). 
Our results are an extension of the results by Colin and Colin (\cite{CC04}) and Gr\"unrock (\cite{Gr}).

Now, we give the main results in the present paper. 
For a Banach space $H$ and $r>0$, we define $B_r(H):=\{ f\in H \,|\, \|f\|_H \le r \}$. 
Furthermore, we put $\theta :=\alpha \beta \gamma (1/\alpha -1/\beta -1/\gamma )$ and $\kappa :=(\alpha -\beta )(\alpha -\gamma )(\beta +\gamma )$. 
Note that if $\alpha$, $\beta$, $\gamma \in \R\backslash \{0\}$ and $\theta \geq0$, then $\kappa \neq 0$. 
\begin{thm}\label{wellposed_1}\kuuhaku \\
{\rm (i)} We assume that $\alpha$, $\beta$, $\gamma \in \R\backslash \{0\}$ satisfy
$\kappa \neq 0$ if $d\geq 4$, and $\theta >0$ if $d=2, 3$. 
Then {\rm (\ref{NLS_sys})} is globally well-posed for small data in $\dot{H}^{s_{c}}$. 
More precisely, there exists $r>0$ such that for all initial data $(u_{0}, v_{0}, w_{0})\in B_{r}(\dot{H}^{s_{c}}\times \dot{H}^{s_{c}}\times \dot{H}^{s_{c}})$, there exists a solution
\[
(u,v,w)\in \dot{X}^{s_{c}}([0,\infty ))\subset C([0,\infty );\dot{H}^{s_{c}})
\]
of the system {\rm (\ref{NLS_sys})} on $(0, \infty )$. 
Such solution is unique in $\dot{X}_{r}^{s_{c}}([0,\infty ))$ which is a closed subset of $\dot{X}^{s_{c}}([0,\infty ))$ 
{\rm (see (\ref{Xs_norm}) and (\ref{Xrs_norm}))}. 
Moreover, the flow map
\[
S_{+}:B_{r}(\dot{H}^{s_{c}}\times \dot{H}^{s_{c}}\times \dot{H}^{s_{c}})\ni (u_{0},v_{0},w_{0})\mapsto (u,v,w)\in \dot{X}^{s_{c}}([0,\infty ))
\]
is Lipschitz continuous. \\
{\rm (ii)} The statement in {\rm (i)} remains valid if we replace the space $\dot{H}^{s_{c}}$, $\dot{X}^{s_{c}}([0,\infty ))$ and $\dot{X}_{r}^{s_{c}}([0,\infty ))$ 
by $H^{s}$, $X^{s}([0,\infty ))$ and $X_{r}^{s}([0,\infty ))$ for $s\geq s_{c}$. 
\end{thm}
\begin{rem}
Due to the time reversibility of the system (\ref{NLS_sys}), the above theorems also hold in corresponding intervals $(-\infty, 0)$. 
We denote the flow map with respect to $(-\infty ,0)$ by $S_{-}$. 
\end{rem}
\begin{cor}\label{sccat}\kuuhaku \\
{\rm (i)} We assume that $\alpha$, $\beta$, $\gamma \in \R\backslash \{0\}$ satisfy $\kappa \neq 0$ if $d\geq 4$, and $\theta >0$ if $d=2, 3$. 
Let $r>0$ be as in Theorem~\ref{wellposed_1}. 
For every $(u_{0},v_{0},w_{0})\in B_{r}(\dot{H}^{s_{c}}\times \dot{H}^{s_{c}}\times \dot{H}^{s_{c}})$, there exists 
$(u_{\pm},v_{\pm},w_{\pm})\in \dot{H}^{s_{c}}\times \dot{H}^{s_{c}}\times \dot{H}^{s_{c}}$ 
such that 
\[
\begin{split}
S_{\pm}(u_{0},v_{0},w_{0})&-(e^{it\alpha \Delta}u_{\pm}, e^{it\beta \Delta}v_{\pm}, e^{it\gamma \Delta}w_{\pm}) 
\rightarrow 0\\
&{\rm in}\ \dot{H^{s_{c}}}\times \dot{H}^{s_{c}}\times \dot{H}^{s_{c}}\ {\rm as}\ t\rightarrow \pm \infty. 
\end{split}
\]
{\rm (ii)} The statement in {\rm (i)} remains valid if we replace the space 
$\dot{H}^{s_{c}}$ by $H^{s}$ for $s\geq s_{c}$.  
\end{cor}
\begin{thm}\label{wellposed_2}Let $\alpha$, $\beta$, $\gamma \in \R\backslash \{0\}$.  \\
{\rm (i)} Let $d\geq 4$. We assume $(\alpha -\gamma )(\beta +\gamma ) \neq 0$ and $s>s_{c}$. 
Then {\rm (\ref{NLS_sys})} is locally well-posed in $H^{s}$. 
More precisely, for any $r>0$ and for all initial data $(u_{0}, v_{0}, w_{0})\in B_{r}(H^{s}\times H^{s}\times H^{s})$, there exist $T=T(r)>0$ and a solution
\[
(u,v,w)\in X^{s}([0,T])\subset C\left([0,T];H^{s}\right)
\]
of the system {\rm (\ref{NLS_sys})} on $(0, T]$. 
Such solution is unique in $X_{r}^{s}([0,T])$ which is a closed subset of $X^{s}([0,T])$. 
Moreover, the flow map
\[
S_{+}:B_{r}(H^{s}\times H^{s}\times H^{s})\ni (u_{0},v_{0},w_{0})\mapsto (u,v,w)\in X^{s}([0,T])
\]
is Lipschitz continuous. \\
{\rm (ii)} Let $d=2$, $3$. We assume $s>s_{c}$ if $\theta >0$, $s\geq 1$ if $\theta \leq 0$ and $\kappa \neq 0$, and $s>1$ if $\alpha =\beta$. 
Then the statement in {\rm (i)} remains valid. \\ 
{\rm (iii)} Let $d=1$. We assume $s\geq 0$ if $\theta > 0$, $s\geq 1$ if $\theta= 0$, and $s\geq 1/2$ if $\theta<0$ and $(\alpha -\gamma )(\beta +\gamma ) \neq 0$.  
Then the statement in {\rm (i)} remains valid.  
\end{thm}
\begin{rem}
For the case $d=1$, $1>s\geq 1/2$, $\theta <0$ and $(\alpha -\gamma )(\beta +\gamma ) \neq 0$, 
we prove the well-posedness as $X^{s}([0,T])=X^{s,b}_{\alpha}([0,T])\times X^{s,b}_{\beta}([0,T])\times X^{s,b}_{\gamma}([0,T])$, 
where $X^{s,b}_{\sigma}$ denotes the standard Bourgain space 
which is the completion of the Schwarz space with respect to the norm 
$||u||_{X^{s,b}_{\sigma}}:=||\langle \xi \rangle^{s}\langle \tau +\sigma \xi^{2}\rangle^{b}\widetilde{u}||_{L^{2}_{\tau \xi}}$ 
{\rm (}see Appendix~\ref{appendix}{\rm )}.   
\end{rem}
System (\ref{NLS_sys}) has the following conservation quantities (see Proposition~\ref{conservation}):
\[
\begin{split}
M(u,v,w)&:=2||u||_{L^{2}_{x}}^{2}+||v||_{L^{2}_{x}}^{2}+||w||_{L^{2}_{x}}^{2},\\
H(u,v,w)&:=\alpha ||\nabla u||_{L^{2}_{x}}^{2}+\beta ||\nabla v||_{L^{2}_{x}}^{2}+\gamma ||\nabla w||_{L^{2}_{x}}^{2}+2{\rm Re}(w ,\nabla (u\cdot \overline{v}))_{L^{2}_{x}}. 
\end{split}
\]
By using the conservation law for $M$ and $H$, we obtain the following result.
\begin{thm}\label{global_extend}\kuuhaku \\
{\rm (i)} Let $d=1$ and assume that $\alpha$, $\beta$, $\gamma \in \R\backslash \{0\}$ satisfy $\theta >0$. 
For every $(u_{0}, v_{0}, w_{0})\in L^{2}\times L^{2}\times L^{2}$, we can extend the local $L^{2}$ solution of {\rm Theorem~\ref{wellposed_2}} globally in time. \\
{\rm (ii)} We assume that $\alpha$, $\beta$, $\gamma \in \R \backslash \{0\}$ have the same sign and satisfy $\kappa \neq 0$ if $d=2$, $3$ and 
 $(\alpha -\gamma )(\beta +\gamma ) \neq 0$ if $d=1$. There exists $r>0$ such that 
for every $(u_{0}, v_{0}, w_{0})\in B_{r}(H^{1}\times H^{1}\times H^{1})$, we can extend the local $H^{1}$ solution of {\rm Theorem~\ref{wellposed_2}} globally in time. 
\end{thm}
While, we obtain the negative result as follows.
\begin{thm}\label{notC2}
Let $d\geq 1$ and $\alpha$, $\beta$, $\gamma \in \R\backslash \{0\}$.  
We assume $s\in \R$ if $(\alpha -\gamma )(\beta +\gamma )=0$, $s<1$ if $\theta =0$, and $s<1/2$ if $\theta <0$. 
Then the flow map of {\rm (\ref{NLS_sys})} is not $C^{2}$ in $H^{s}$. 
\end{thm}
Furthermore, for the equation (\ref{ddqdnls}), we obtain the following result.
\begin{thm}\label{ddqdnls_wp}
Let $d\geq 2$. Then, the equation {\rm (\ref{ddqdnls})} is 
globally well-posed for small data in $\dot{H}^{s_{c}}$ {\rm (resp.} $H^{s}$ for $s\geq s_{c}${\rm )} 
and the solution converges to a free solution in $\dot{H}^{s_{c}}$ {\rm (resp.} $H^{s}$ for $s\geq s_{c}${\rm )}  asymptotically in time. 
\end{thm}
\begin{rem}
The results by Gr\"unrock (\cite{Gr}) are not contained the critical case $s=s_{c}$ 
and global property of the solution. 
In this sense, Theorem~\ref{ddqdnls_wp} is the extension of the results by Gr\"unrock (\cite{Gr}).  
\end{rem}
The main tools of our results are $U^{p}$ space and $V^{p}$ space which are applied to prove 
the well-posedness and scattering for KP-II equation at the scaling critical regularity by Hadac, Herr and Koch (\cite{HHK09}, \cite{HHK10}). 
After their work, $U^{p}$ space and $V^{p}$ space are used to prove the well-posedness of the 3D periodic quintic nonlinear Schr\"odinger equation 
at the scaling critical regularity by Herr, Tataru and Tzvetkov (\cite{HTT11}) 
and to prove the well-posedness and the scattering of the quadratic Klein-Gordon system at the scaling critical regularity by Schottdorf (\cite{Sc}).  
\kuuhaku \\

\noindent {\bf Notation.} 
We denote the spatial Fourier transform by\ \ $\widehat{\cdot}$\ \ or $\F_{x}$, 
the Fourier transform in time by $\F_{t}$ and the Fourier transform in all variables by\ \ $\widetilde{\cdot}$\ \ or $\F_{tx}$. 
For $\sigma \in \R$, the free evolution $e^{it\sigma \Delta}$ on $L^{2}$ is given as a Fourier multiplier
\[
\F_{x}[e^{it\sigma \Delta}f](\xi )=e^{-it\sigma |\xi |^{2}}\widehat{f}(\xi ). 
\]
We will use $A\lesssim B$ to denote an estimate of the form $A \le CB$ for some constant $C$ and write $A \sim B$ to mean $A \lesssim B$ and $B \lesssim A$. 
We will use the convention that capital letters denote dyadic numbers, e.g. $N=2^{n}$ for $n\in \Z$ and for a dyadic summation we write
$\sum_{N}a_{N}:=\sum_{n\in \Z}a_{2^{n}}$ and $\sum_{N\geq M}a_{N}:=\sum_{n\in \Z, 2^{n}\geq M}a_{2^{n}}$ for brevity. 
Let $\chi \in C^{\infty}_{0}((-2,2))$ be an even, non-negative function such that $\chi (t)=1$ for $|t|\leq 1$. 
We define $\psi (t):=\chi (t)-\chi (2t)$ and $\psi_{N}(t):=\psi (N^{-1}t)$. Then, $\sum_{N}\psi_{N}(t)=1$ whenever $t\neq 0$. 
We define frequency and modulation projections
\[
\widehat{P_{N}u}(\xi ):=\psi_{N}(\xi )\widehat{u}(\xi ),\ 
\widehat{P_{0}u}(\xi ):=\psi_{0}(\xi )\widehat{u}(\xi ),\ 
\widetilde{Q_{M}^{\sigma}u}(\tau ,\xi ):=\psi_{M}(\tau +\sigma |\xi|^{2})\widetilde{u}(\tau ,\xi ),
\]
where $\psi_{0}:=1-\sum_{N\geq 1}\psi_{N}$ Furthermore, we define $Q_{\geq M}^{\sigma}:=\sum_{N\geq M}Q_{N}^{\sigma}$ and $Q_{<M}:=Id -Q_{\geq M}$. 

The rest of this paper is planned as follows.
In Section 2, we will give the definition and properties of the $U^{p}$ space and $V^{p}$ space. 
In Sections 3, 4 and 5, we will give the bilinear and trilinear estimates which will be used to prove the well-posedness.
In Section 6, we will give the proof of the well-posedness and the scattering (Theorems~\ref{wellposed_1}, ~\ref{wellposed_2}, ~\ref{ddqdnls_wp} and Corollary~\ref{sccat}). 
In Section 7, we will give the a priori estimates and show Theorem~\ref{global_extend}. 
In Section 8, we will give the proof of $C^{2}$-ill-posedness (Theorem~\ref{notC2}). 
In Appendix A, we will give the proof of the bilinear estimates for the standard $1$-dimensional 
Bourgain norm 
under the condition 
$(\alpha-\gamma)(\beta+\gamma)\neq 0$ and $\alpha \beta \gamma (1/\alpha-1/\beta-1/\gamma)\neq 0$.
%

\section{$U^{p}$, $V^{p}$ spaces  and their properties \label{func_sp}}
In this section, we define the $U^{p}$ space and the $V^{p}$ space, 
and introduce the properties of these spaces 
which are proved by Hadac, Herr and Koch (\cite{HHK09}, \cite{HHK10}). 

We define the set of finite partitions $\ZZ$ as
\[
\ZZ :=\left\{ \{t_{k}\}_{k=0}^{K}|K\in \N , -\infty <t_{0}<t_{1}<\cdots <t_{K}\leq \infty \right\}
\]
and if $t_{K}=\infty$, we put $v(t_{K}):=0$ for all functions $v:\R \rightarrow L^{2}$. 
\begin{defn}\label{upsp}
Let $1\leq p <\infty$. For $\{t_{k}\}_{k=0}^{K}\in \ZZ$ and $\{\phi_{k}\}_{k=0}^{K-1}\subset L^{2}$ with 
$\sum_{k=0}^{K-1}||\phi_{k}||_{L^{2}}^{p}=1$ we call the function $a:\R\rightarrow L^{2}$ 
given by
\[
a(t)=\sum_{k=1}^{K}\ee_{[t_{k-1},t_{k})}(t)\phi_{k-1}
\]
a ``$U^{p}${\rm -atom}''. 
Furthermore, we define the atomic space 
\[
U^{p}:=\left\{ \left. u=\sum_{j=1}^{\infty}\lambda_{j}a_{j}
\right| a_{j}:U^{p}{\rm -atom},\ \lambda_{j}\in \C \ {\rm such\ that}\  \sum_{j=1}^{\infty}|\lambda_{j}|<\infty \right\}
\]
with the norm
\[
||u||_{U^{p}}:=\inf \left\{\sum_{j=1}^{\infty}|\lambda_{j}|\left|u=\sum_{j=1}^{\infty}\lambda_{j}a_{j},\ 
a_{j}:U^{p}{\rm -atom},\ \lambda_{j}\in \C\right.\right\}.
\]
\end{defn}
\begin{defn}\label{vpsp}
Let $1\leq p <\infty$. We define the space of the bounded $p$-variation 
\[
V^{p}:=\{ v:\R\rightarrow L^{2}|\ ||v||_{V^{p}}<\infty \}
\]
with the norm
\begin{equation}\label{vpnorm}
||v||_{V^{p}}:=\sup_{\{t_{k}\}_{k=0}^{K}\in \ZZ}\left(\sum_{k=1}^{K}||v(t_{k})-v(t_{k-1})||_{L^{2}}^{p}\right)^{1/p}.
\end{equation}
Likewise, let $V^{p}_{-, rc}$ denote the closed subspace of all right-continuous functions $v\in V^{p}$ with 
$\lim_{t\rightarrow -\infty}v(t)=0$, endowed with the same norm {\rm (\ref{vpnorm})}.
\end{defn}
\begin{prop}\label{upvpprop}
Let $1\leq p<q<\infty$. \\
{\rm (i)} $U^{p}$, $V^{p}$ and $V^{p}_{-, rc}$ are Banach spaces. \\ 
{\rm (ii)} Every $u\in U^{p}$ is right-continuous as $u:\R \rightarrow L^{2}$. \\
{\rm (iii)} For Every $u\in U^{p}$, $\lim_{t\rightarrow -\infty}u(t)=0$ and $\lim_{t\rightarrow \infty}u(t)$ exists in $L^{2}$. \\
{\rm (iv)} For Every $v\in V^{p}$, $\lim_{t\rightarrow -\infty}v(t)$ and $\lim_{t\rightarrow \infty}v(t)$ exist in $L^{2}$. \\
{\rm (v)} The embeddings $U^{p}\hookrightarrow V^{p}_{-,rc}\hookrightarrow U^{q}\hookrightarrow L^{\infty}_{t}(\R ;L^{2}_{x}(\R^{d}))$ are continuous. 
\end{prop}
For Proposition~\ref{upvpprop} and its proof, see Propositions 2.2, 2.4 and Corollary 2.6 in \cite{HHK09}.
\begin{thm}\label{duality}
Let $1<p<\infty$ and $1/p+1/p'=1$. 
If $u\in V^{1}_{-,rc}$ be absolutely continuous on every compact intervals, then
\[
||u||_{U^{p}}=\sup_{v\in V^{p'}, ||v||_{V^{p'}}=1}\left|\int_{-\infty}^{\infty}(u'(t),v(t))_{L^{2}(\R^{d})}dt\right|.
\]
\end{thm}
For Theorem~\ref{duality} and its proof, see Theorem 2.8 Proposition 2.10 and Remark 2.11 in \cite{HHK09}.
\begin{defn}
Let $1\leq p<\infty$. For $\sigma \in \R$, we define
\[
U^{p}_{\sigma}:=\{ u:\R\rightarrow L^{2}|\ e^{-it \sigma \Delta}u\in U^{p}\}
\]
with the norm $||u||_{U^{p}_{\sigma}}:=||e^{-it \sigma \Delta}u||_{U^{p}}$, 
\[
V^{p}_{\sigma}:=\{ v:\R\rightarrow L^{2}|\ e^{-it \sigma \Delta}v\in V^{p}\}
\]
with the norm $||v||_{V^{p}_{\sigma}}:=||e^{-it \sigma \Delta}v||_{V^{p}}$ 
and similarly the closed subspace $V^{p}_{-,rc,\sigma}$.
\end{defn}
\begin{rem}
We note that $||\overline{u}||_{U^{p}_{\sigma}}=||u||_{U^{p}_{-\sigma}}$ and $||\overline{v}||_{V^{p}_{\sigma}}=||v||_{V^{p}_{-\sigma}}$. 
\end{rem}
\begin{prop}\label{projest}
Let $1< p<\infty$. We have
\begin{align}
&||Q_{M}^{\sigma}u||_{L_{t}^{p}L_{x}^{2}}\lesssim M^{-1/p}||u||_{V^{p}_{\sigma}},\ ||Q_{\geq M}^{\sigma}u||_{L_{t}^{p}L_{x}^{2}}\lesssim M^{-1/p}||u||_{V^{p}_{\sigma}},\label{highMproj}\\
&||Q_{<M}^{\sigma}u||_{V^{p}_{\sigma}}\lesssim ||u||_{V^{p}_{\sigma}},\ \ ||Q_{\geq M}^{\sigma}u||_{V^{p}_{\sigma}}\lesssim ||u||_{V^{p}_{\sigma}},\label{Vproj}\\
&||Q_{<M}^{\sigma}u||_{U^{p}_{\sigma}}\lesssim ||u||_{U^{p}_{\sigma}},\ \ ||Q_{\geq M}^{\sigma}u||_{U^{p}_{\sigma}}\lesssim ||u||_{U^{p}_{\sigma}}. \label{Uproj}
\end{align}
\end{prop}
For Proposition~\ref{projest} and its proof, 
see Corollary 2.18 in \cite{HHK09}. (\ref{highMproj}) for $p\neq 2$ can be proved by the same way. 
\begin{prop}\label{multiest}
Let 
\[
T_{0}:L^{2}(\R^{d})\times \cdots \times L^{2}(\R^{d})\rightarrow L^{1}_{loc}(\R^{d})
\]
be a $m$-linear operator and $I\subset \R$ be an interval. Assume that for some $1\leq p, q< \infty$
\[
||T_{0}(e^{it\sigma_{1}\Delta}\phi_{1},\cdots ,e^{it\sigma_{m}\Delta}\phi_{m})||_{L^{p}_{t}(I :L^{q}_{x}(\R^{d}))}\lesssim \prod_{i=1}^{m}||\phi_{i}||_{L^{2}(\R^{d})}.
\]
Then, there exists $T:U^{p}_{\sigma_{1}}\times \cdots \times U^{p}_{\sigma_{m}}\rightarrow L^{p}_{t}(I ;L^{q}_{x}(\R^{d}))$ satisfying
\[
|T(u_{1},\cdots ,u_{m})||_{L^{p}_{t}(I ;L^{q}_{x}(\R^{d}))}\lesssim \prod_{i=1}^{m}||u_{i}||_{U^{p}_{\sigma_{i}}}
\]
such that $T(u_{1},\cdots ,u_{m})(t)(x)=T_{0}(u_{1}(t),\cdots ,u_{m}(t))(x)$ a.e.
\end{prop}
For Proposition~\ref{multiest} and its proof, see Proposition 2.19 in \cite{HHK09}.
\begin{prop}[Strichartz estimate]\label{Stri_est}
Let $\sigma \in \R\backslash \{0\}$ and $(p,q)$ be an admissible pair of exponents for the Schr\"odinger equation, i.e. $2\leq q\leq 2d/(d-2)$ 
{\rm (}$2\leq q< \infty$ if $d=2$, $2\leq q\leq \infty$ if $d=1${\rm )}, $2/p =d(1/2-1/q)$. Then, we have
\[
||e^{it\sigma \Delta}\varphi ||_{L_{t}^{p}L_{x}^{q}}\lesssim ||\varphi ||_{L^{2}_{x}}
\]
for any $\varphi \in L^{2}(\R^{d})$. 
\end{prop}
By Proposition~\ref{multiest} and ~\ref{Stri_est}, we have following:
\begin{cor}\label{UV_Stri}
Let $\sigma \in \R\backslash \{0\}$ and $(p,q)$ be an admissible pair of exponents for the Schr\"odinger equation, i.e. $2\leq q\leq 2d/(d-2)$ 
{\rm (}$2\leq q< \infty$ if $d=2$, $2\leq q\leq \infty$ if $d=1${\rm )}, $2/p =d(1/2-1/q)$. Then, we have
\begin{align}
&||u||_{L_{t}^{p}L_{x}^{q}}\lesssim ||u||_{U_{\sigma}^{p}},\ \ u\in U^{p}_{\sigma},\label{U_Stri}\\
&||u||_{L_{t}^{p}L_{x}^{q}}\lesssim ||u||_{V_{\sigma}^{\widetilde{p}}},\ \ u\in V^{\widetilde{p}}_{\sigma}, \ (1\leq \widetilde{p}<p).\label{V_Stri}
\end{align}
\end{cor}
\begin{prop}\label{intpol}
Let $q>1$, $E$ be a Banach space and $T:U^{q}_{\sigma}\rightarrow E$ be a bounded, linear operator 
with $||Tu||_{E}\leq C_{q}||u||_{U^{q}_{\sigma}}$ for all $u\in U^{q}_{\sigma}$.  
In addition, assume that for some $1\leq p<q$ there exists $C_{p}\in (0,C_{q}]$ such that the estimate $||Tu||_{E}\leq C_{p}||u||_{U^{p}_{\sigma}}$ holds true for all $u\in U^{p}_{\sigma}$. Then, $T$ satisfies the estimate
\[
||Tu||_{E}\lesssim C_{p}\left( 1+\ln \frac{C_{q}}{C_{p}}\right) ||u||_{V^{p}_{\sigma}},\ \ u\in V^{p}_{-,rc,\sigma}, 
\]
where implicit constant depends only on $p$ and $q$.
\end{prop}
For Proposition~\ref{intpol} and its proof, 
see Proposition 2.20 in \cite{HHK09}.
\section{Bilinear Strichartz estimates \label{be_for_2d}}
%
%
In this section, implicit constants in $\ll$ actually depend on $\sigma_{1}$, $\sigma_{2}$. 
\begin{lemm}\label{L2be}
Let $d\in \N$, $s_{c}=d/2-1$, $b>1/2$ and $\sigma_{1}$, $\sigma_{2}\in \R \backslash \{0\}$. 
For any dyadic numbers $L$, $H\in 2^{\Z}$ with $L\ll H$, we have
\begin{equation}\label{L2be_est}
||(P_{H}u_{1})(P_{L}u_{2})||_{L^{2}_{tx}}\lesssim L^{s_{c}}\left(\frac{L}{H}\right)^{1/2}||P_{H}u_{1}||_{X_{\sigma_{1}}^{0,b}}||P_{L}u_{2}||_{X_{\sigma_{2}}^{0,b}}, 
\end{equation}
where $||u||_{X_{\sigma}^{0,b}}:=||\langle \tau +\sigma |\xi|^{2}\rangle^{b}\widetilde{u}||_{L^{2}_{\tau \xi}}$.
\end{lemm}
\begin{proof}
For the case $d=2$ and $(\sigma_{1},\sigma_{2})=(1,\pm 1)$, the estimate (\ref{L2be_est}) is proved by 
Colliander, Delort, Kenig, and Staffilani (\cite{CDKS01}, Lemma 1). 
The proof for general case as following is similar to their argument. 

We put $g_{1}(\tau_{1},\xi_{1}):=\langle \tau_{1}+\sigma_{1}|\xi_{1}|^{2}\rangle^{b}\widetilde{P_{H}u_{1}}(\tau_{1},\xi_{1})$,\ 
$g_{2}(\tau_{2},\xi_{2}):=\langle \tau_{2}+\sigma_{2}|\xi_{2}|^{2}\rangle^{b}\widetilde{P_{L}u_{2}}(\tau_{2},\xi_{2})$ 
and $A_{N}:=\{\xi \in \R^{d} |N/2\leq |\xi |\leq 2N\}$
for a dyadic number $N$. 
By the Plancherel's theorem and the duality argument, it is enough to prove the estimate
\begin{equation}\label{Iestimate}
\begin{split}
I&:=\left|\int_{\R}\int_{\R}\int_{A_{L}}\int_{A_{H}}f(\tau_{1}+\tau_{2}, \xi_{1}+\xi_{2})
\frac{g_{1}(\tau_{1},\xi_{1})}{\langle \tau_{1}+\sigma_{1}|\xi_{1}|^{2}\rangle^{b}}
\frac{g_{2}(\tau_{2},\xi_{2})}{\langle \tau_{2}+\sigma_{2}|\xi_{2}|^{2}\rangle^{b}}d\xi_{1}d\xi_{2}d\tau_{1}d\tau_{2}\right|\\
&\lesssim \frac{L^{(d-1)/2}}{H^{1/2}}||f||_{L^{2}_{\tau \xi}}||g_{1}||_{L^{2}_{\tau \xi}}||g_{2}||_{L^{2}_{\tau \xi}}
\end{split}
\end{equation}
for $f\in L^{2}_{\tau \xi}$. 
We change the variables $(\tau_{1}, \tau_{2})\mapsto (\theta_{1}, \theta_{2})$ as 
$\theta_{i}=\tau_{i}+\sigma_{i}|\xi_{i}|^{2}$ $(i=1,2)$
and put 
\[
\begin{split}
F(\theta_{1},\theta_{2}, \xi_{1},\xi_{2})&:=f(\theta_{1}+\theta_{2}-\sigma_{1}|\xi_{1}|^{2}-\sigma_{2}|\xi_{2}|^{2}, \xi_{1}+\xi_{2}),\\
G_{i}(\theta_{i},\xi_{i})&:=g_{i}(\theta_{i}-\sigma_{i}|\xi_{i}|^{2},\xi_{i}),\ \ (i=1,2). 
\end{split}
\]
Then, we have
\[
\begin{split}
I&\leq \int_{\R}\int_{\R}\frac{1}{\langle \theta_{1}\rangle^{b}\langle \theta_{2}\rangle^{b}}
\left(\int_{A_{L}}\int_{A_{H}}|F(\theta_{1}, \theta_{2}, \xi_{1}, \xi_{2})G_{1}(\theta_{1}, \xi_{1})G_{2}(\theta_{2}, \xi_{2})|d\xi_{1}d\xi_{2}\right) d\theta_{1}d\theta_{2}\\
&\lesssim \int_{\R}\int_{\R}\frac{1}{\langle \theta_{1}\rangle^{b}\langle \theta_{2}\rangle^{b}}
\left(\int_{A_{L}}\int_{A_{H}}|F(\theta_{1}, \theta_{2}, \xi_{1}, \xi_{2})|^{2}d\xi_{1}d\xi_{2}\right)^{1/2}||G_{1}(\theta_{1}, \cdot )||_{L^{2}_{\xi}}||G_{2}(\theta_{2}, \cdot )||_{L^{2}_{\xi}}d\theta_{1}d\theta_{2}
\end{split}
\]
by the Cauchy-Schwarz inequality. 
For $1\leq j\leq d$, we put
\[
A_{H}^{j}:=\{\xi_{1}=(\xi_{1}^{(1)},\cdots, \xi_{1}^{(d)})\in \R^{d}|\ H/2\leq |\xi_{1}|\leq 2H,\ |\xi_{1}^{(j)}|\geq H/(2\sqrt{d})\}
\]
and 
\[
K_{j}(\theta_{1},\theta_{2}):=\int_{A_{L}}\int_{A_{H}^{j}}|F(\theta_{1}, \theta_{2}, \xi_{1}, \xi_{2})|^{2}d\xi_{1}d\xi_{2}. 
\]
We consider only the estimate for $K_{1}$. The estimates for other $K_{j}$ are obtained by the same way. 

Assume $d\geq 2$. 
By changing the variables $(\xi_{1}, \xi_{2})=(\xi_{1}^{(1)},\cdots, \xi_{1}^{(d)}, \xi_{2}^{(1)},\cdots, \xi_{2}^{(d)}) \mapsto (\mu, \nu, \eta )$ as
\begin{equation}\label{ch_var}
\begin{cases}
\mu =\theta_{1}+\theta_{2}-\sigma_{1}|\xi_{1}|^{2}-\sigma_{2}|\xi_{2}|^{2}\in \R , \\
\nu =\xi_{1}+\xi_{2}\in \R^{d},\\
\eta=(\xi_{2}^{(2)}\cdots, \xi_{2}^{(d)})\in \R^{d-1}, 
\end{cases}
\end{equation}
we have 
\[
d\mu d\nu d\eta=2|\sigma_{1}\xi_{1}^{(1)}-\sigma_{2}\xi_{2}^{(1)}|d\xi_{1}d\xi_{2}
\]
and
\[
F(\theta_{1}, \theta_{2}, \xi_{1}, \xi_{2})=f(\mu, \nu ).
\]
We note that $|\sigma_{1}\xi_{1}^{(1)}-\sigma_{2}\xi_{2}^{(1)}|\sim H$ for any $(\xi_{1}, \xi_{2})\in A_{H}^{1}\times A_{L}$ with $L\ll H$. 
Furthermore, $\xi_{2}\in A_{L}$ implies that $\eta \in [-2L, 2L]^{d-1}$.  Therefore,  we obtain
\[
K_{1}(\theta_{1},\theta_{2})\lesssim 
\frac{1}{H}\int_{[-2L, 2L]^{d-1}} \int_{\R^{d}} \int_{\R} |f(\mu,\nu)|^{2} d\mu d\nu d\eta 
\sim\frac{L^{d-1}}{H}||f||_{L^{2}_{\tau \xi}}^{2}.
\]
As a result, we have
\[
\begin{split}
I&\lesssim \int_{\R}\int_{\R}\frac{1}{\langle \theta_{1}\rangle^{b}\langle \theta_{1}\rangle^{b}}\left(\sum_{j=1}^{d}K_{j}(\theta_{1},\theta_{2})\right)^{1/2}||G_{1}(\theta_{1}, \cdot )||_{L^{2}_{\xi}}||G_{2}(\theta_{2}, \cdot )||_{L^{2}_{\xi}}d\theta_{1}d\theta_{2}\\
&\lesssim \frac{L^{(d-1)/2}}{H^{1/2}}||f||_{L^{2}_{\tau \xi}}||g_{1}||_{L^{2}_{\tau \xi}}||g_{2}||_{L^{2}_{\tau \xi}} \\
\end{split}
\]
by the Cauchy-Schwarz inequality and changing the variables $(\theta_{1}, \theta_{2})\mapsto (\tau_{1}, \tau_{2})$ as $\theta_{i}=\tau_{i}+\sigma_{i}|\xi_{i}|^{2}$ $(i=1,2)$. 

For $d=1$, we obtain the same result by changing the variables $(\xi_{1}, \xi_{2}) \mapsto (\mu, \nu)$ as
$\mu =\theta_{1}+\theta_{2}-\sigma_{1}|\xi_{1}|^{2}-\sigma_{2}|\xi_{2}|^{2}$, $\nu =\xi_{1}+\xi_{2}$ instead of (\ref{ch_var}). 
\end{proof}
\begin{cor}\label{UVbilinear}
Let $d\in \N$, $s_{c}=d/2-1$ and $\sigma_{1}$, $\sigma_{2}\in \R \backslash \{0\}$. \\
{\rm (i)}\ If $d\geq 2$, then for any dyadic numbers $L$, $H\in 2^{\Z}$ with $L\ll H$, we have
\begin{align}
&||(P_{H}u_{1})(P_{L}u_{2})||_{L^{2}_{tx}}
\lesssim L^{s_{c}}\left(\frac{L}{H}\right)^{1/2}||P_{H}u_{1}||_{U_{\sigma_{1}}^{2}}||P_{L}u_{2}||_{U_{\sigma_{2}}^{2}},\label{Ubilinear}\\
&||(P_{H}u_{1})(P_{L}u_{2})||_{L^{2}_{tx}}
\lesssim L^{s_{c}}\left(\frac{L}{H}\right)^{1/2}
\left( 1+\ln \frac{H}{L}\right)^{2}||P_{H}u_{1}||_{V_{\sigma_{1}}^{2}}||P_{L}u_{2}||_{V_{\sigma_{2}}^{2}}.\label{Vbilinear}
\end{align}
{\rm (ii)}\ If $d=1$, then for any dyadic numbers $L$, $H\in 2^{\Z}$ with $L\ll H$ , we have 
\begin{align}
&||(P_{H}u_{1})(P_{L}u_{2})||_{L^{2}([0,1]\times \R)}
\lesssim \frac{1}{H^{1/2}}||P_{H}u_{1}||_{U_{\sigma_{1}}^{2}}||P_{L}u_{2}||_{U_{\sigma_{2}}^{2}},\label{Ubilinear_1d}\\
&||(P_{H}u_{1})(P_{L}u_{2})||_{L^{2}([0,1]\times \R)}
\lesssim \min\left\{L^{1/6}, \frac{\left( 1+\ln H\right)^{2}}{H^{1/2}}
\right\}||P_{H}u_{1}||_{V_{\sigma_{1}}^{2}}||P_{L}u_{2}||_{V_{\sigma_{2}}^{2}}.\label{Vbilinear_1d}
\end{align}
\end{cor}
\begin{proof}
To obtain (\ref{Ubilinear}) and (\ref{Ubilinear_1d}), we use the argument of the proof of Corollary 2.21\ (27) in \cite{HHK09}.
Let $\phi_{1}$, $\phi_{2}\in L^{2}(\R^{d})$ and define $\phi_{j}^{\lambda}(x):=\phi_{j}(\lambda x)$ $(j=1,2)$
for $\lambda \in \R$. By using the rescaling $(t,x)\mapsto (\lambda^{2}t, \lambda x)$, we have
\[
\begin{split}
&||P_{H}(e^{it\sigma_{1}\Delta}\phi_{1})P_{L}(e^{it\sigma_{2}\Delta}\phi_{2})||_{L^{2}([-T,T]\times \R^{d})}\\
&=\lambda^{s_{c}+2}||P_{\lambda H}(e^{it\sigma_{1}\Delta}\phi_{1}^{\lambda})P_{\lambda L}(e^{it\sigma_{2}\Delta}\phi_{2}^{\lambda})||_{L^{2}([-\lambda^{-2}T,\lambda^{-2}T]\times \R^{d})}. 
\end{split}
\]
Therefore by putting $\lambda =\sqrt{T}$ and Lemma~\ref{L2be}, we have
\[
\begin{split}
&||P_{H}(e^{it\sigma_{1}\Delta}\phi_{1})P_{L}(e^{it\sigma_{2}\Delta}\phi_{2})||_{L^{2}([-T,T]\times \R^{d})}\\
&\lesssim \sqrt{T}^{2(s_{c}+1)}L^{s_{c}}\left(\frac{L}{H}\right)^{1/2}||P_{\sqrt{T}H}\phi_{1}^{\sqrt{T}}||_{L^{2}_{x}}||P_{\sqrt{T}L}\phi_{2}^{\sqrt{T}}||_{L^{2}_{x}}\\
&=L^{s_{c}}\left(\frac{L}{H}\right)^{1/2}||P_{H}\phi_{1}||_{L^{2}_{x}}||P_{L}\phi_{2}||_{L^{2}_{x}}.
\end{split}
\]
Let $T\rightarrow \infty$, then we obtain 
\[
||P_{H}(e^{it\sigma_{1}\Delta}\phi_{1})P_{L}(e^{it\sigma_{2}\Delta}\phi_{2})||_{L^{2}_{tx}}
\lesssim L^{s_{c}}\left(\frac{L}{H}\right)^{1/2}||P_{H}\phi_{1}||_{L^{2}_{x}}||P_{L}\phi_{2}||_{L^{2}_{x}}
\]
and (\ref{Ubilinear}), (\ref{Ubilinear_1d}) follow from proposition~\ref{multiest}. 

To obtain (\ref{Vbilinear}) and (\ref{Vbilinear_1d}), we first prove the $U^{4}$ estimate for $d\geq 2$ and $U^{8}$ estimate for $d=1$. 
Assume $d\geq 2$. By the Cauchy-Schwarz inequality, the Sobolev embedding $\dot{W}^{s_{c}, 2d/(d-1)}(\R^{d})\hookrightarrow L^{2d}(\R^{d})$ and (\ref{U_Stri}), we have
\begin{equation}\label{U4_est}
\begin{split}
||(P_{H}u_{1})(P_{L}u_{2})||_{L^{2}_{tx}}&\lesssim L^{s_{c}}||P_{H}u_{1}||_{L^{4}_{t}L^{2d/(d-1)}_{x}}||P_{L}u_{2}||_{L^{4}_{t}L^{2d/(d-1)}_{x}}\\
&\lesssim L^{s_{c}}||P_{H}u_{1}||_{U^{4}_{\sigma_{1}}}||P_{L}u_{2}||_{U^{4}_{\sigma_{2}}}
\end{split}
\end{equation}
for any dyadic numbers $L$, $H\in 2^{\Z}$. 
While if $d=1$, then by the H\"older's inequality and (\ref{U_Stri}), we have
\begin{equation}\label{U8_est}
\begin{split}
||(P_{H}u_{1})(P_{L}u_{2})||_{L^{2}([0,1]\times \R )}
&\leq ||\ee_{[0,1)}||_{L^{4}_{t}}||P_{H}u_{1}||_{L^{8}_{t}L^{4}_{x}}||P_{L}u_{2}||_{L^{8}_{t}L^{4}_{x}}\\
&\lesssim ||P_{H}u_{1}||_{U^{8}_{\sigma_{1}}}||P_{L}u_{2}||_{U^{8}_{\sigma_{2}}} 
\end{split}
\end{equation}
for any dyadic numbers $L$, $H\in 2^{\Z}$. 
We use the interpolation between (\ref{Ubilinear}) and (\ref{U4_est}) via 
Proposition~\ref{intpol}. Then, we get (\ref{Vbilinear}) 
by the same argument of the proof of Corollary 2.21\ (28) in \cite{HHK09}. 
The estimate (\ref{Vbilinear_1d}) follows from 
\begin{equation}\label{U12_est}
\begin{split}
||(P_{H}u_{1})(P_{L}u_{2})||_{L^{2}([0,1]\times \R )}
&\leq ||\ee_{[0,1)}||_{L^{3}_{t}}L^{1/6}||P_{H}u_{1}||_{L^{12}_{t}L^{3}_{x}}||P_{L}u_{2}||_{L^{12}_{t}L^{3}_{x}}\\
&\lesssim L^{1/6}||P_{H}u_{1}||_{V^{2}_{\sigma_{1}}}||P_{L}u_{2}||_{V^{2}_{\sigma_{2}}}. 
\end{split}
\end{equation}
and the interpolation between (\ref{Ubilinear_1d}) and (\ref{U8_est}),  
where we used the H\"older's inequality, the Sobolev embedding $\dot{W}^{1/6,3}(\R )\hookrightarrow L^{6}(\R )$ 
and (\ref{V_Stri}) to obtain (\ref{U12_est}). 
\end{proof}
%
%
\section{Time global estimates for $d\geq 2$ \label{global_tri_est_2d}}
In this and next section, implicit constants in $\ll$ actually depend on $\sigma_{1}$, $\sigma_{2}$, $\sigma_{3}$.  
\begin{lemm}\label{modul_est}
Let $d\in \N$. We assume that $\sigma_{1}$, $\sigma_{2}$, $\sigma_{3} \in \R \backslash \{0\}$ satisfy $(\sigma_{1}+\sigma_{2})(\sigma_{2}+\sigma_{3})(\sigma_{3}+\sigma_{1})\neq 0$ and $(\tau_{1},\xi_{1})$, $(\tau_{2}, \xi_{2})$, $(\tau_{3}, \xi_{3})\in \R\times \R^{d}$ satisfy $\tau_{1}+\tau_{2}+\tau_{3}=0$, $\xi_{1}+\xi_{2}+\xi_{3}=0$.\\ 
{\rm (i)} If there exist $1\leq i,j\leq 3$ such that $|\xi_{i}|\ll |\xi_{j}|$, then we have
\begin{equation}\label{modulation_est}
\max_{1\leq j\leq 3}|\tau_{j}+\sigma_{j}|\xi_{j}|^{2}|
\gtrsim \max_{1\leq j\leq 3}|\xi_{j}|^{2}. 
\end{equation}
{\rm (ii)} If $\sigma_{1}\sigma_{2}\sigma_{3}(1/\sigma_{1}+1/\sigma_{2}+1/\sigma_{3})>0$, then we have {\rm (\ref{modulation_est})}. 
\end{lemm}
\begin{proof}
By the triangle inequality and the completing the square, we have
\begin{equation}\label{com_squ}
\begin{split}
M_{0}:=&\max_{1\leq j\leq 3}|\tau_{j}+\sigma_{j}|\xi_{j}|^{2}|\\
&\gtrsim |\sigma_{1}|\xi_{1}|^{2}+\sigma_{2}|\xi_{2}|^{2}+\sigma_{3}|\xi_{3}|^{2}|\\
&=|(\sigma_{1}+\sigma_{3})|\xi_{1}|^{2}+2\sigma_{3}\xi_{1}\cdot \xi_{2}+(\sigma_{2}+\sigma_{3})|\xi_{2}|^{2}|\\
&=|\sigma_{1}+\sigma_{3}|\left| \left| \xi_{1}+\frac{\sigma_{3}}{\sigma_{1}+\sigma_{3}}\xi_{2}\right|^{2}+\frac{\sigma_{1}\sigma_{2}\sigma_{3}}{(\sigma_{1}+\sigma_{3})^{2}}\left(\frac{1}{\sigma_{1}}+\frac{1}{\sigma_{2}}+\frac{1}{\sigma_{3}}\right)|\xi_{2}|^{2}\right| .
\end{split}
\end{equation}
We first prove {\rm (i)}. By the symmetry, we can assume $|\xi_{1}|\sim |\xi_{3}|\gtrsim |\xi_{2}|$. 
If $|\xi_{1}|\gg |\xi_{2}|$, then we have $M_{0}\gtrsim |\xi_{1}|^{2}\sim \max_{1\leq j\leq 3}|\xi_{j}|^{2}$ by (\ref{com_squ}). 
Next, we prove {\rm (ii)}. By the symmetry, we can assume $|\xi_{1}|\sim |\xi_{2}|\gtrsim |\xi_{3}|$. 
If $\sigma_{1}\sigma_{2}\sigma_{3}(1/\sigma_{1}+1/\sigma_{2}+1/\sigma_{3})>0$, 
then we have $M_{0}\gtrsim |\xi_{2}|^{2}\sim \max_{1\leq j\leq 3}|\xi_{j}|^{2}$ by (\ref{com_squ}). 
\end{proof}
In the following Propositions and Corollaries in this and next section, we assume $P_{N_{1}}u_{1}\in V^{2}_{-,rc,\sigma_{1}}$, $P_{N_{2}}u_{2}\in V^{2}_{-,rc,\sigma_{2}}$ and $P_{N_{3}}u_{3}\in V^{2}_{-,rc,\sigma_{3}}$ for each $N_{1}$, $N_{2}$, $N_{3}\in 2^{\Z}$. 
Propositions~\ref{HL_est}, ~\ref{HH_est} and its proofs are based on Proposition\ 3.1 in \cite{HHK09}.  
\begin{prop}\label{HL_est}
Let $d\geq 2$, $s_{c}=d/2-1$, $0<T\leq \infty$ and $\sigma_{1}$, $\sigma_{2}$, $\sigma_{3}\in \R \backslash \{ 0\}$ satisfy $(\sigma_{1}+\sigma_{2})(\sigma_{2}+\sigma_{3})(\sigma_{3}+\sigma_{1})\neq 0$. 
For any dyadic numbers $N_{2}$, $N_{3}\in 2^{\Z}$ with $N_{2}\sim N_{3}$, we have
\begin{equation}\label{hl}
\begin{split}
&\left|\sum_{N_{1}\ll N_{2}}N_{\max}\int_{0}^{T}\int_{\R^{d}}(P_{N_{1}}u_{1})(P_{N_{2}}u_{2})(P_{N_{3}}u_{3})dxdt\right|\\
&\lesssim 
\left(\displaystyle \sum_{N_{1}\ll N_{2}}N_{1}^{2s_{c}}||P_{N_{1}}u_{1}||_{V^{2}_{\sigma_{1}}}^{2}\right)^{1/2}||P_{N_{2}}u_{2}||_{V^{2}_{\sigma_{2}}}||P_{N_{3}}u_{3}||_{V^{2}_{\sigma_{3}}},
\end{split}
\end{equation}
where $\displaystyle N_{\max}:=\max_{1\leq j\leq 3}N_{j}$. 
\end{prop}
\begin{proof} 
We define $f_{j,N_{j},T}:=\ee_{[0,T)}P_{N_{j}}u_{j}$\ $(j=1,2,3)$. For sufficiently large constant C, we put 
$M:=C^{-1}N_{\max}^{2}$ and decompose $Id=Q^{\sigma_{j}}_{<M}+Q^{\sigma_{j}}_{\geq M}$\ $(j=1,2,3)$. 
We divide the integrals on the left-hand side of (\ref{hl}) into eight piece of the form 
\begin{equation}\label{piece_form}
\int_{\R}\int_{\R^{d}}(Q_{1}^{\sigma_{1}}f_{1,N_{1},T})(Q_{2}^{\sigma_{2}}f_{2,N_{2},T})(Q_{3}^{\sigma_{3}}f_{3,N_{3},T})dxdt
\end{equation}
with $Q_{j}^{\sigma_{j}}\in \{Q_{\geq M}^{\sigma_{j}}, Q_{<M}^{\sigma_{j}}\}$\ $(j=1,2,3)$. 
By the Plancherel's theorem, we have
\[
(\ref{piece_form})
= c\int_{\tau_{1}+\tau_{2}+\tau_{3}=0}\int_{\xi_{1}+\xi_{2}+\xi_{3}=0}\FT[Q_{1}^{\sigma_{1}}f_{1,N_{1},T}](\tau_{1},\xi_{1})\FT[Q_{2}^{\sigma_{2}}f_{2,N_{2},T}](\tau_{2}, \xi_{2})\FT[Q_{3}^{\sigma_{3}}f_{3,N_{3},T}](\tau_{3}, \xi_{3}),
\]
where $c$ is a constant. Therefore, Lemma~\ref{modul_est} {\rm (i)} implies that
\[
\int_{\R}\int_{\R^{d}}(Q_{<M}^{\sigma_{1}}f_{1,N_{1},T})(Q_{<M}^{\sigma_{2}}f_{2,N_{2},T})(Q_{<M}^{\sigma_{3}}f_{3,N_{3},T})dxdt=0
\]
when $N_{1}\ll N_{2}$. 
So, let us now consider the case that $Q_{j}^{\sigma_{j}}=Q_{\geq M}^{\sigma_{j}}$ for some $1\leq j\leq 3$. 

First, we consider the case $Q_{1}^{\sigma_{1}}=Q_{\geq M}^{\sigma_{1}}$. 
By the H\"older's inequality and the Sobolev embedding $\dot{H}^{s_{c}}(\R^{d})\hookrightarrow L^{d}(\R^{d})$, we have
\begin{equation}\label{HL_1_est}
\begin{split}
&\left|\sum_{N_{1}\ll N_{2}}N_{\max}\int_{\R}\int_{\R^{d}}(Q_{\geq M}^{\sigma_{1}}f_{1,N_{1},T})(Q_{2}^{\sigma_{2}}f_{2,N_{2},T})(Q_{3}^{\sigma_{3}}f_{3,N_{3},T})dxdt\right|\\
&\lesssim \left|\left|\sum_{N_{1}\ll N_{2}}N_{\max}|\nabla |^{s_{c}}Q_{\geq M}^{\sigma_{1}}f_{1,N_{1},T}\right|\right|_{L^{2}_{tx}}
||Q_{2}^{\sigma_{2}}f_{2,N_{2},T}||_{L_{t}^{4}L_{x}^{2d/(d-1)}}||Q_{3}^{\sigma_{3}}f_{3,N_{3},T}||_{L_{t}^{4}L_{x}^{2d/(d-1)}}. 
\end{split}
\end{equation}
Furthermore, by the $L^{2}$ orthogonality and (\ref{highMproj}) with $p=2$, we have
\[
\left|\left|\sum_{N_{1}\ll N_{2}}N_{\max}|\nabla |^{s_{c}}Q_{\geq M}^{\sigma_{1}}f_{1,N_{1},T}\right|\right|_{L^{2}_{tx}}
\lesssim \left(\sum_{N_{1}\ll N_{2}}N_{\max}^{2}N_{1}^{2s_{c}}M^{-1}||f_{1,N_{1},T}||_{V_{\sigma_{1}}^{2}}^{2}\right)^{1/2}
\]
While by (\ref{V_Stri}) and (\ref{Vproj}), we have
\[
||Q_{2}^{\sigma_{2}}f_{2,N_{2},T}||_{L_{t}^{4}L_{x}^{2d/(d-1)}}\lesssim ||f_{2,N_{2},T}||_{V_{\sigma_{2}}^{2}},\ 
||Q_{3}^{\sigma_{3}}f_{3,N_{3},T}||_{L_{t}^{4}L_{x}^{2d/(d-1)}}\lesssim ||f_{3,N_{3},T}||_{V_{\sigma_{3}}^{2}}.
\]
Therefore, we obtain 
\[
\begin{split}
&\left|\sum_{N_{1}\ll N_{2}}N_{\max}\int_{\R}\int_{\R^{d}}(Q_{\geq M}^{\sigma_{1}}f_{1,N_{1},T})(Q_{2}^{\sigma_{2}}f_{2,N_{2},T})(Q_{3}^{\sigma_{3}}f_{3,N_{3},T})dxdt\right|\\
&\lesssim 
\left(\displaystyle \sum_{N_{1}\ll N_{2}}N_{1}^{2s_{c}}||P_{N_{1}}u_{1}||_{V^{2}_{\sigma_{1}}}^{2}\right)^{1/2}||P_{N_{2}}u_{2}||_{V^{2}_{\sigma_{2}}}||P_{N_{3}}u_{3}||_{V^{2}_{\sigma_{3}}},
\end{split}
\]
since $M\sim N_{\max}^{2}$ and $||\ee_{[0,T)}f||_{V^{2}_{\sigma}}\lesssim ||f||_{V^{2}_{\sigma}}$ for any $\sigma \in \R$ and any $T\in (0,\infty]$. 

Next, we consider the case $Q_{3}^{\sigma_{3}}=Q_{\geq M}^{\sigma_{3}}$. 
By the Cauchy-Schwarz inequality, we have
\[
\begin{split}
&\left|\sum_{N_{1}\ll N_{2}}N_{\max}\int_{\R}\int_{\R^{d}}(Q_{1}^{\sigma_{1}}f_{1,N_{1},T})(Q_{2}^{\sigma_{2}}f_{2,N_{2},T})(Q_{\geq M}^{\sigma_{3}}f_{3,N_{3},T})dxdt\right|\\
&\leq \sum_{N_{1}\ll N_{2}}N_{\max}||(Q_{1}^{\sigma_{1}}f_{1,N_{1},T})(Q_{2}^{\sigma_{2}}f_{2,N_{2},T})||_{L^{2}_{tx}}||Q_{\geq M}^{\sigma_{3}}f_{3,N_{3},T}||_{L^{2}_{tx}}. 
\end{split}
\]
Furthermore, by (\ref{highMproj}) with $p=2$, we have
\begin{equation}\label{HL_3_est}
||Q_{\geq M}^{\sigma_{3}}f_{3,N_{3},T}||_{L^{2}_{tx}}
\lesssim M^{-1/2}||f_{3,N_{3},T}||_{V^{2}_{\sigma_{3}}}.
\end{equation}
While by (\ref{Vbilinear}), (\ref{Vproj})
and the Cauchy-Schwarz inequality for the dyadic sum, we have
\begin{equation}\label{HL_4_est}
\begin{split}
&\sum_{N_{1}\ll N_{2}}||(Q_{1}^{\sigma_{1}}f_{1,N_{1},T})(Q_{2}^{\sigma_{2}}f_{2,N_{2},T})||_{L^{2}_{tx}}\\
&\lesssim \sum_{N_{1}\ll N_{2}}N_{1}^{s_{c}}\left(\frac{N_{1}}{N_{2}}\right)^{1/4}
||Q_{1}^{\sigma_{1}}f_{1,N_{1},T}||_{V_{\sigma_{1}}^{2}}||Q_{2}^{\sigma_{2}}f_{2,N_{2},T}||_{V_{\sigma_{2}}^{2}}\\
&\lesssim \left(\sum_{N_{1}\ll N_{2}}N_{1}^{2s_{c}}||f_{1,N_{1},T}||_{V_{\sigma_{1}}^{2}}^{2}\right)^{1/2}
||f_{2,N_{2},T}||_{V_{\sigma_{2}}^{2}}.
\end{split}
\end{equation}
Therefore, we obtain 
\[
\begin{split}
&\left|\sum_{N_{1}\ll N_{2}}N_{\max}\int_{\R}\int_{\R^{d}}(Q_{1}^{\sigma_{1}}f_{1,N_{1},T})(Q_{2}^{\sigma_{2}}f_{2,N_{2},T})(Q_{\geq M}^{\sigma_{3}}f_{3,N_{3},T})dxdt\right|\\
&\lesssim 
\left(\displaystyle \sum_{N_{1}\ll N_{2}}N_{1}^{2s_{c}}||P_{N_{1}}u_{1}||_{V^{2}_{\sigma_{1}}}^{2}\right)^{1/2}||P_{N_{2}}u_{2}||_{V^{2}_{\sigma_{2}}}||P_{N_{3}}u_{3}||_{V^{2}_{\sigma_{3}}},
\end{split}
\]
since $M\sim N_{\max}^{2}$ and $||\ee_{[0,T)}f||_{V^{2}_{\sigma}}\lesssim ||f||_{V^{2}_{\sigma}}$ for any $\sigma \in \R$ and any $T\in (0,\infty]$. 

For the case $Q_{2}^{\sigma_{2}}=Q_{\geq M}^{\sigma_{2}}$ is proved in exactly same way as the case $Q_{3}^{\sigma_{3}}=Q_{\geq M}^{\sigma_{3}}$.
\end{proof}
\begin{prop}\label{HH_est}
Let $d\geq 2$, $s_{c}=d/2-1$, $s\geq 0$, $0<T\leq \infty$ and $\sigma_{1}$, $\sigma_{2}$, $\sigma_{3}\in \R \backslash \{ 0\}$ satisfy $(\sigma_{1}+\sigma_{2})(\sigma_{2}+\sigma_{3})(\sigma_{3}+\sigma_{1})\neq 0$. 
For any dyadic numbers $N_{1}$, $N_{2}\in 2^{\Z}$ with $N_{1}\sim N_{2}$, we have
\begin{equation}\label{hh}
\begin{split}
&\left(\sum_{N_{3}\ll N_{2}}N_{3}^{2s}
\sup_{||u_{3}||_{V^{2}_{\sigma_{3}}}=1}\left| N_{\max}\int_{0}^{T}\int_{\R^{d}}(P_{N_{1}}u_{1})(P_{N_{2}}u_{2})(P_{N_{3}}u_{3})dxdt\right|^{2}
\right)^{1/2}\\
&\lesssim  N_{1}^{s_{c}}||P_{N_{1}}u_{1}||_{V^{2}_{\sigma_{1}}}N_{2}^{s}||P_{N_{2}}u_{2}||_{V^{2}_{\sigma_{2}}}, 
\end{split}
\end{equation}
where $\displaystyle N_{\max}:=\max_{1\leq j\leq 3}N_{j}$. 
\end{prop}
\begin{proof} 
We define 
$f_{j,N_{j},T}:=\ee_{[0,T)}P_{N_{j}}u_{j}$\ $(j=1,2,3)$. For sufficiently large constant C, we put 
$M:=C^{-1}N_{\max}^{2}$ and decompose $Id=Q^{\sigma_{j}}_{<M}+Q^{\sigma_{j}}_{\geq M}$\ $(j=1,2,3)$. 
We divide the integrals on the left-hand side of (\ref{hh}) into eight piece of the form 
\begin{equation}\label{piece_form_hh}
\int_{\R}\int_{\R^{d}}(Q_{1}^{\sigma_{1}}f_{1,N_{1},T})(Q_{2}^{\sigma_{2}}f_{2,N_{2},T})(Q_{3}^{\sigma_{3}}f_{3,N_{3},T})dxdt
\end{equation}
with $Q_{j}^{\sigma_{j}}\in \{Q_{\geq M}^{\sigma_{j}}, Q_{<M}^{\sigma_{j}}\}$\ $(j=1,2,3)$. 
By the same argument of the proof of Proposition~\ref{HL_est}, we consider only the case that $Q_{j}^{\sigma_{j}}=Q_{\geq M}^{\sigma_{j}}$ for some $1\leq j\leq 3$. 

First, we consider the case $Q_{1}^{\sigma_{1}}=Q_{\geq M}^{\sigma_{1}}$. 
By the Cauchy-Schwarz inequality, we have
\[
\begin{split}
&\left|\int_{\R}\int_{\R^{d}}(Q_{\geq M}^{\sigma_{1}}f_{1,N_{1},T})(Q_{2}^{\sigma_{2}}f_{2,N_{2},T})(Q_{3}^{\sigma_{3}}f_{3,N_{3},T})dxdt\right|\\
&\leq ||Q_{\geq M}^{\sigma_{1}}f_{1,N_{1},T}||_{L^{2}_{tx}}||(Q_{2}^{\sigma_{2}}f_{2,N_{2},T})(Q_{3}^{\sigma_{3}}f_{3,N_{3},T})||_{L^{2}_{tx}}. 
\end{split}
\]
Furthermore by (\ref{highMproj}) with $p=2$, we have
\begin{equation}\label{HH_1_est}
||Q_{\geq M}^{\sigma_{1}}f_{1,N_{1},T}||_{L^{2}_{tx}}
\lesssim M^{-1/2}||f_{1,N_{1},T}||_{V^{2}_{\sigma_{1}}}.
\end{equation}
While by (\ref{Vbilinear}) and (\ref{Vproj}), we have
\begin{equation}\label{HH_2_est}
\begin{split}
||(Q_{2}^{\sigma_{2}}f_{2,N_{2},T})(Q_{3}^{\sigma_{3}}f_{3,N_{3},T})||_{L^{2}_{tx}}
&\lesssim N_{3}^{s_{c}}\left(\frac{N_{3}}{N_{2}}\right)^{1/4}
||f_{2,N_{2},T}||_{V_{\sigma_{2}}^{2}}||f_{3,N_{3},T}||_{V_{\sigma_{3}}^{2}}
\end{split}
\end{equation}
when $N_{3}\ll N_{2}$. Therefore, we obtain
\[
\begin{split}
&\sum_{N_{3}\ll N_{2}}N_{3}^{2s}\sup_{||u_{3}||_{V^{2}_{\sigma_{3}}}=1}\left|N_{\max}\int_{\R}\int_{\R^{d}}(Q_{\geq M}^{\sigma_{1}}f_{1,N_{1},T})(Q_{2}^{\sigma_{2}}f_{2,N_{2},T})(Q_{3}^{\sigma_{3}}f_{3,N_{3},T})dxdt\right|^{2}\\
&\lesssim N_{1}^{2s_{c}}||P_{N_{1}}u_{1}||_{V^{2}_{\sigma_{1}}}^{2}N_{2}^{2s}||P_{N_{2}}u_{2}||_{V^{2}_{\sigma_{2}}}^{2}
\end{split}
\]
by $M\sim N_{\max}^{2}$, $N_{1}\sim N_{2}$  and $||\ee_{[0,T)}f||_{V^{2}_{\sigma}}\lesssim ||f||_{V^{2}_{\sigma}}$ for any $\sigma \in \R$ and any $T\in (0,\infty]$. 

Next, we consider the case $Q_{3}^{\sigma_{3}}=Q_{\geq M}^{\sigma_{3}}$. 
We define $\widetilde{P}_{N_{3}}=P_{N_{3}/2}+P_{N_{3}}+P_{2N_{3}}$. 
By the Cauchy-Schwarz inequality, we have
\[
\begin{split}
&\left|\int_{\R}\int_{\R^{d}}(Q_{1}^{\sigma_{1}}f_{1,N_{1},T})(Q_{2}^{\sigma_{2}}f_{2,N_{2},T})(Q_{\geq M}^{\sigma_{3}}f_{3,N_{3},T})dxdt\right|\\
&\lesssim ||\widetilde{P}_{N_{3}}( (Q_{1}^{\sigma_{1}}f_{1,N_{1},T})(Q_{2}^{\sigma_{2}}f_{2,N_{2},T}))||_{L^{2}_{tx}}||Q_{\geq M}^{\sigma_{3}}f_{3,N_{3},T}||_{L^{2}_{tx}}
\end{split}
\]
since $P_{N_{3}}=\widetilde{P}_{N_{3}}P_{N_{3}}$. Furthermore, by (\ref{highMproj}) with $p=2$, we have
\begin{equation}\label{HH_3_est}
||Q_{\geq M}^{\sigma_{3}}f_{3,N_{3},T}||_{L^{2}_{tx}}
\lesssim M^{-1/2}||f_{3,N_{3},T}||_{V^{2}_{\sigma_{3}}}.
\end{equation}
Therefore, we obtain
\begin{equation}\label{HH_4_est}
\begin{split}
&\sum_{N_{3}\ll N_{2}}N_{3}^{2s}\sup_{||u_{3}||_{V^{2}_{\sigma_{3}}}=1}\left|N_{\max}\int_{\R}\int_{\R^{d}}(Q_{1}^{\sigma_{1}}f_{1,N_{1},T})(Q_{2}^{\sigma_{2}}f_{2,N_{2},T})(Q_{\geq M}^{\sigma_{3}}f_{3,N_{3},T})dxdt\right|^{2}\\
&\lesssim \sum_{N_{3}\ll N_{2}}N_{3}^{2s}N_{\max}^{2}M^{-1}||\widetilde{P}_{N_{3}}( (Q_{1}^{\sigma_{1}}f_{1,N_{1},T})(Q_{2}^{\sigma_{2}}f_{2,N_{2},T}))||_{L^{2}_{tx}}^{2}\\
&\lesssim N_{2}^{2s}||(Q_{1}^{\sigma_{1}}f_{1,N_{1},T})(Q_{2}^{\sigma_{2}}f_{2,N_{2},T})||_{L^{2}_{tx}}^{2}\\
&\lesssim N_{1}^{2s_{c}}||P_{N_{1}}u_{1}||_{V^{2}_{\sigma_{1}}}^{2}N_{2}^{2s}||P_{N_{2}}u_{2}||_{V^{2}_{\sigma_{2}}}^{2}
\end{split}
\end{equation}
by $M\sim N_{\max}^{2}$, $N_{1}\sim N_{2}$, $L^{2}$-orthogonality, (\ref{U4_est}), the embedding $V^{2}_{-,rc}\hookrightarrow U^{4}$, (\ref{Vproj})
and $||\ee_{[0,T)}f||_{V^{2}_{\sigma}}\lesssim ||f||_{V^{2}_{\sigma}}$ for any $\sigma \in \R$ and any $T\in (0,\infty]$. 

For the case $Q_{2}^{\sigma_{2}}=Q_{\geq M}^{\sigma_{2}}$ is proved in exactly same way as the case $Q_{1}^{\sigma_{1}}=Q_{\geq M}^{\sigma_{1}}$. 
\end{proof}
\begin{prop}\label{HHH_est}
Let $s_{c}=d/2-1$ and , $0<T\leq \infty$. \\
{\rm (i)} Let $d\geq 4$. 
For any $\sigma_{1}$, $\sigma_{2}$, $\sigma_{3} \in \R\backslash \{0\}$ and any dyadic numbers $N_{1}$, $N_{2}$, $N_{3}\in 2^{\Z}$ , we have
\begin{equation}\label{hhh}
\begin{split}
&\left|N_{\max}\int_{0}^{T}\int_{\R^{d}}(P_{N_{1}}u_{1})(P_{N_{2}}u_{2})(P_{N_{3}}u_{3})dxdt\right|\\
&\lesssim 
N_{\rm max}^{s_{c}}||P_{N_{1}}u_{1}||_{V^{2}_{\sigma_{1}}}||P_{N_{2}}u_{2}||_{V^{2}_{\sigma_{2}}}||P_{N_{3}}u_{3}||_{V^{2}_{\sigma_{3}}},
\end{split}
\end{equation}
where $\displaystyle N_{\max}:=\max_{1\leq j\leq 3}N_{j}.$\\
{\rm (ii)} Let $d=2$, $3$ and $\sigma_{1}\sigma_{2}\sigma_{3}(1/\sigma_{1}+1/\sigma_{2}+1/\sigma_{3})>0$.
For any dyadic numbers $N_{1}$, $N_{2}$, $N_{3}\in 2^{\Z}$, we have {\rm (\ref{hhh})}. 
\end{prop}
\begin{proof}First, we consider the case $d\geq 4$. 
By the H\"older's inequality, the Sobolev embedding $\dot{W}^{s_{c}-1, 6d/(3d-4)}(\R^{d})\hookrightarrow  L^{3d/4}(\R^{d})$ and
(\ref{V_Stri}), we have
\[
\begin{split}
({\rm L.H.S\ of\ (\ref{hhh})})
&\lesssim  N_{\max}||P_{N_{1}u_{1}}||_{L^{3}_{t}L^{6d/(3d-4)}_{x}}||P_{N_{2}u_{2}}||_{L^{3}_{t}L^{6d/(3d-4)}_{x}}|||\nabla |^{s_{c}-1} P_{N_{3}u_{3}}||_{_{L^{3}_{t}L^{6d/(3d-4)}_{x}}}\\
&\lesssim 
N_{\max}^{s_{c}}||P_{N_{1}u_{1}}||_{V^{2}_{\sigma_{1}}}||P_{N_{2}u_{2}}||_{V^{2}_{\sigma_{2}}}||P_{N_{3}u_{3}}||_{V^{2}_{\sigma_{3}}}.
\end{split}
\]
\kuuhaku 

Next, we consider the case $d=2$, $3$ and $\sigma_{1}\sigma_{2}\sigma_{3}(1/\sigma_{1}+1/\sigma_{2}+1/\sigma_{3})>0$. We define 
$f_{j,N_{j},T}:=\ee_{[0,T)}P_{N_{j}}u_{j}$\ $(j=1,2,3)$. For sufficiently large constant C, we put 
$M:=C^{-1}N_{\max}^{2}$ and decompose $Id=Q^{\sigma_{j}}_{<M}+Q^{\sigma_{j}}_{\geq M}$\ $(j=1,2,3)$. 
We divide the integral on the left-hand side of (\ref{hhh}) into eight piece of the form 
\begin{equation}\label{piece_form_2}
\int_{\R}\int_{\R^{d}}(Q_{1}^{\sigma_{1}}f_{1,N_{1},T})(Q_{2}^{\sigma_{2}}f_{2,N_{2},T})(Q_{3}^{\sigma_{3}}f_{3,N_{3},T})dxdt
\end{equation}
with $Q_{j}^{\sigma_{j}}\in \{Q_{\geq M}^{\sigma_{j}}, Q_{<M}^{\sigma_{j}}\}$\ $(j=1,2,3)$. 
Since $\sigma_{1}\sigma_{2}\sigma_{3}(1/\sigma_{1}+1/\sigma_{2}+1/\sigma_{3})>0$, Lemma~\ref{modul_est} {\rm (ii)} implies that
\[
\int_{\R}\int_{\R^{d}}(Q_{<M}^{\sigma_{1}}f_{1,N_{1},T})(Q_{<M}^{\sigma_{2}}f_{2,N_{2},T})(Q_{<M}^{\sigma_{3}}f_{3,N_{3},T})dxdt=0
\]
for any $N_{1}$, $N_{2}$, $N_{3}\in 2^{\Z}$. 
So, let us now consider the case that $Q_{j}^{\sigma_{j}}=Q_{\geq M}^{\sigma_{j}}$ for some $1\leq j\leq 3$. 
We consider only for the case $Q_{1}^{\sigma_{1}}=Q_{\geq M}^{\sigma_{1}}$ since 
for the other cases is same manner.

By the Cauchy-Schwarz inequality, we have
\[
\begin{split}
&\left|\int_{\R}\int_{\R^{d}}(Q_{\geq M}^{\sigma_{1}}f_{1,N_{1},T})(Q_{2}^{\sigma_{2}}f_{2,N_{2},T})(Q_{3}^{\sigma_{3}}f_{3,N_{3},T})dxdt\right|\\
&\leq ||Q_{\geq M}^{\sigma_{1}}f_{1,N_{1},T}||_{L^{2}_{tx}}||(Q_{2}^{\sigma_{2}}f_{2,N_{2},T})(Q_{3}^{\sigma_{3}}f_{3,N_{3},T})||_{L^{2}_{tx}}.
\end{split}
\]
Furthermore by (\ref{highMproj}) with $p=2$, we have
\begin{equation}\label{HHH_1_est}
\begin{split}
||Q_{\geq M}^{\sigma_{1}}f_{1,N_{1},T}||_{L^{2}_{tx}}
&\lesssim M^{-1/2}||f_{1,N_{1},T}||_{V_{\sigma_{1}}^{2}}. 
\end{split}
\end{equation}
While by (\ref{U4_est}), the embedding $V^{2}_{-,rc}\hookrightarrow U^{4}$ and (\ref{Vproj}), we have
\begin{equation}\label{HHH_2_est}
||(Q_{2}^{\sigma_{2}}f_{2,N_{2},T})(Q_{3}^{\sigma_{3}}f_{3,N_{3},T})||_{L^{2}_{tx}}
\lesssim N_{\rm max}^{s_{c}}
||f_{2,N_{2},T}||_{V_{\sigma_{2}}^{2}}||f_{3,N_{3},T}||_{V_{\sigma_{3}}^{2}}.
\end{equation}
Therefore, we obtain 
\[
\begin{split}
&\left|N_{\max}\int_{\R}\int_{\R^{d}}(Q_{\geq M}^{\sigma_{1}}f_{1,N_{1},T})(Q_{2}^{\sigma_{2}}f_{2,N_{2},T})(Q_{3}^{\sigma_{3}}f_{3,N_{3},T})dxdt\right|\\
&\lesssim 
N_{\rm max}^{s_{c}}||P_{N_{1}}u_{1}||_{V^{2}_{\sigma_{1}}}||P_{N_{2}}u_{2}||_{V^{2}_{\sigma_{2}}}||P_{N_{3}}u_{3}||_{V^{2}_{\sigma_{3}}},
\end{split}
\]
since $M\sim N_{\max}^{2}$ and $||\ee_{[0,T)}f||_{V^{2}_{\sigma}}\lesssim ||f||_{V^{2}_{\sigma}}$ for any $\sigma \in \R$ and any $T\in (0,\infty]$. 
\end{proof}
Proposition~\ref{HH_est} and Proposition~\ref{HHH_est} imply the following: 
\begin{cor}\label{HH-HL_est}
Let $\sigma_{1}$, $\sigma_{2}$, $\sigma_{3} \in \R\backslash \{0\}$ satisfy $(\sigma_{1}+\sigma_{2})(\sigma_{2}+\sigma_{3})(\sigma_{3}+\sigma_{1})\neq 0$ if $d\geq 4$, and $\sigma_{1}\sigma_{2}\sigma_{3}(1/\sigma_{1}+1/\sigma_{2}+1/\sigma_{3})>0$ if $d=2$, $3$. 
Then the estimate {\rm (\ref{hh})} holds if we replace $\sum_{N_{3}\ll N_{2}}$ by $\sum_{N_{3}\lesssim N_{2}}$. 
\end{cor} 
%
%
\section{Time local estimates \label{local_tri_est_2d}}
\begin{prop}\label{HL_est_sub}
Let $s>s_{c}$ $(=d/2-1)$, $0<T< \infty$ if $d\geq 2$ and $s\geq 0$, $T=1$ if $d=1$. 
We assume $\sigma_{1}$, $\sigma_{2}$, $\sigma_{3}\in \R \backslash \{ 0\}$ satisfy $(\sigma_{1}+\sigma_{2})(\sigma_{2}+\sigma_{3})(\sigma_{3}+\sigma_{1})\neq 0$. 
For any dyadic numbers $N_{2}$, $N_{3}\in 2^{\Z}$ with $N_{2}\sim N_{3}$, we have
\begin{equation}\label{hl_sub}
\begin{split}
&\left|\sum_{N_{1}\ll N_{2}}N_{\max}\int_{0}^{T}\int_{\R^{d}}(P_{N_{1}}u_{1})(P_{N_{2}}u_{2})(P_{N_{3}}u_{3})dxdt\right|\\
&\lesssim 
T^{\delta}\left(\displaystyle \sum_{N_{1}\ll N_{2}}(N_{1}\vee 1)^{2s}||P_{N_{1}}u_{1}||_{V^{2}_{\sigma_{1}}}^{2}\right)^{1/2}||P_{N_{2}}u_{2}||_{V^{2}_{\sigma_{2}}}||P_{N_{3}}u_{3}||_{V^{2}_{\sigma_{3}}}
\end{split}
\end{equation}
for some $\delta >0$, where $\displaystyle N_{\max}:=\max_{1\leq j\leq 3}N_{j}$. 
\end{prop}
\begin{proof}
First, we assume $d\geq 2$. We choose $\delta >0$ satisfying $\delta <(s-s_{c})/2$ and $\delta \ll 1$. 
In the proof of proposition~\ref{HL_est}, for L.H.S of (\ref{HL_1_est}), 
we use the Sobolev embedding $\dot{H}^{s_{c}+2\delta}\hookrightarrow L^{d/(1-2\delta )}$ 
instead of $\dot{H}^{s_{c}}\hookrightarrow L^{d}$ . Then we have
\[
\begin{split}
&\left|\sum_{N_{1}\ll N_{2}}N_{\max}\int_{\R}\int_{\R^{d}}(Q_{\geq M}^{\sigma_{1}}f_{1,N_{1},T})(Q_{2}^{\sigma_{2}}f_{2,N_{2},T})(Q_{3}^{\sigma_{3}}f_{3,N_{3},T})dxdt\right|\\
&\lesssim ||\ee_{[0,T)}||_{L^{1/\delta }_{t}}\left|\left|\sum_{N_{1}\ll N_{2}}N_{\max}|\nabla |^{s_{c}+2\delta}Q_{\geq M}^{\sigma_{1}}f_{1,N_{1},T}\right|\right|_{L^{2}_{tx}}
||Q_{2}^{\sigma_{2}}f_{2,N_{2},T}||_{L_{t}^{p}L_{x}^{q}}||Q_{3}^{\sigma_{3}}f_{3,N_{3},T}||_{L_{t}^{p}L_{x}^{q}}\\
&\leq T^{\delta }\left|\left|\sum_{N_{1}\ll N_{2}}N_{\max}\langle \nabla \rangle^{s}Q_{\geq M}^{\sigma_{1}}f_{1,N_{1},T}\right|\right|_{L^{2}_{tx}}
||Q_{2}^{\sigma_{2}}f_{2,N_{2},T}||_{L_{t}^{p}L_{x}^{q}}||Q_{3}^{\sigma_{3}}f_{3,N_{3},T}||_{L_{t}^{p}L_{x}^{q}}
\end{split}
\]
with $(p, q)=(4/(1-2\delta ), 2d/(d-1+2\delta))$ which is the admissible pair of the Strichartz estimate. 
Furthermore for L.H.S of (\ref{HL_3_est}), 
we use the H\"older's inequality and (\ref{highMproj}) with $p=2/(1-2\delta )$ instead of $p=2$. 
Then we have
\[
||Q_{\geq M}^{\sigma_{3}}f_{3,N_{3},T}||_{L^{2}_{tx}}
\leq ||\ee_{[0,T)}||_{L^{1/\delta}_{t}}||Q_{\geq M}^{\sigma_{3}}f_{3,N_{3},T}||_{L_{t}^{2/(1-2\delta )}L_{x}^{2}}
\lesssim T^{\delta}M^{-(1-2\delta )/2}||f_{3,N_{3},T}||_{V^{2}_{\sigma_{3}}}.
\]
For the other part, by the same way of the proof of proposition~\ref{HL_est}, we obtain (\ref{hl_sub}). 

Next, we assume $d=1$. In the proof of proposition~\ref{HL_est}, for L.H.S of (\ref{HL_1_est}), 
we use the H\"older's inequality as follows: 
\[
\begin{split}
&\left|\sum_{N_{1}\ll N_{2}}N_{\max}\int_{\R}\int_{\R}(Q_{\geq M}^{\sigma_{1}}f_{1,N_{1},T})(Q_{2}^{\sigma_{2}}f_{2,N_{2},T})(Q_{3}^{\sigma_{3}}f_{3,N_{3},T})dxdt\right|\\
&\lesssim ||\ee_{[0,T)}||_{L^{4}_{t}}\left|\left|\sum_{N_{1}\ll N_{2}}N_{\max}Q_{\geq M}^{\sigma_{1}}f_{1,N_{1},T}\right|\right|_{L^{2}_{tx}}
||Q_{2}^{\sigma_{2}}f_{2,N_{2},T}||_{L_{t}^{8}L_{x}^{4}}||Q_{3}^{\sigma_{3}}f_{3,N_{3},T}||_{L_{t}^{8}L_{x}^{4}}. 
\end{split}
\]
We note that $(8,4)$ is the admissible pair of the Strichartz estimate for $d=1$.
Furthermore for the first inequality in (\ref{HL_4_est}), we use (\ref{Vbilinear_1d}) instead of (\ref{Vbilinear}). 
For the other part, by the same way of the proof of proposition~\ref{HL_est}, we obtain (\ref{hl_sub}) with $T=1$.  
\end{proof}
\begin{prop}\label{HH_est_sub}
Let $s>s_{c}$ $(=d/2-1)$, $0<T< \infty$ if $d\geq 2$ and $s\geq 0$, $T=1$ if $d=1$. 
We assume $\sigma_{1}$, $\sigma_{2}$, $\sigma_{3}\in \R \backslash \{ 0\}$ satisfy $(\sigma_{1}+\sigma_{2})(\sigma_{2}+\sigma_{3})(\sigma_{3}+\sigma_{1})\neq 0$. 
For any dyadic numbers $N_{1}$, $N_{2}\in 2^{\Z}$ with $N_{1}\sim N_{2}$, we have
\begin{equation}\label{hh_sub}
\begin{split}
&\left(\sum_{N_{3}\ll N_{2}}N_{3}^{2s}
\sup_{||u_{3}||_{V^{2}_{\sigma_{3}}}=1}\left|N_{\max}\int_{0}^{T}\int_{\R^{d}}(P_{N_{1}}u_{1})(P_{N_{2}}u_{2})(P_{N_{3}}u_{3})dxdt\right|^{2}
\right)^{1/2}\\
&\lesssim  T^{\delta}(N_{1}\vee 1)^{s}||P_{N_{1}}u_{1}||_{V^{2}_{\sigma_{1}}}(N_{2}\vee 1)^{s}||P_{N_{2}}u_{2}||_{V^{2}_{\sigma_{2}}}. 
\end{split}
\end{equation}
for some $\delta >0$, where $\displaystyle N_{\max}:=\max_{1\leq j\leq 3}N_{j}$. 
\end{prop}
\begin{proof}
First, we assume $d\geq 2$. We choose $\delta >0$ satisfying $\delta <(s-s_{c})/2$ and $\delta \ll 1$. 
In the proof of proposition~\ref{HH_est}, for L.H.S of (\ref{HH_1_est}) and (\ref{HH_3_est}), 
we use the H\"older's inequality and (\ref{highMproj}) with $p=2/(1-2\delta )$ instead of $p=2$. 
Then we have
\[
\begin{split}
&||Q_{\geq M}^{\sigma_{1}}f_{1,N_{1},T}||_{L^{2}_{tx}}
\leq ||\ee_{[0,T)}||_{L^{1/\delta}_{t}}||Q_{\geq M}^{\sigma_{1}}f_{1,N_{1},T}||_{L_{t}^{2/(1-2\delta )}L_{x}^{2}}
\lesssim T^{\delta}M^{-(1-2\delta )/2}||f_{1,N_{1},T}||_{V^{2}_{\sigma_{1}}},\\
&||Q_{\geq M}^{\sigma_{3}}f_{3,N_{3},T}||_{L^{2}_{tx}}
\leq ||\ee_{[0,T)}||_{L^{1/\delta}_{t}}||Q_{\geq M}^{\sigma_{3}}f_{3,N_{3},T}||_{L_{t}^{2/(1-2\delta )}L_{x}^{2}}
\lesssim T^{\delta}M^{-(1-2\delta )/2}||f_{3,N_{3},T}||_{V^{2}_{\sigma_{3}}}.
\end{split}
\]
For the other part, by the same way of the proof of proposition~\ref{HH_est}, we obtain (\ref{hh_sub}).  

Next, we assume $d=1$. 
In the proof of proposition~\ref{HH_est}, for L.H.S of (\ref{HH_2_est}), 
we use (\ref{Vbilinear_1d}) instead of (\ref{Vbilinear}) 
and for the third inequality in (\ref{HH_4_est}), we use (\ref{U8_est}) and $V_{-,rc}^{2}\hookrightarrow U^{8}$ instead of 
(\ref{U4_est}) and $V_{-,rc}^{2}\hookrightarrow U^{4}$. 
For the other part, by the same way of the proof of proposition~\ref{HH_est}, we obtain (\ref{hh_sub}) with $T=1$.   
\end{proof}
\begin{prop}\label{HHH_est_sub}\kuuhaku \\
{\rm (i)} Let $d\geq 4$, $s>s_{c}$ and $0<T<\infty$. 
For any $\sigma_{1}$, $\sigma_{2}$, $\sigma_{3}\in \R \backslash \{ 0\}$, any dyadic numbers $N_{1}$, $N_{2}$, $N_{3}\in 2^{\Z}$ 
 and $1\leq j\leq 3$, we have
\begin{equation}\label{hhh_sub}
\begin{split}
&\left|N_{j}\int_{0}^{T}\int_{\R^{d}}(P_{N_{1}}u_{1})(P_{N_{2}}u_{2})(P_{N_{3}}u_{3})dxdt\right|\\
&\lesssim 
T^{\delta}(N_{j}\vee 1)^{s}||P_{N_{1}}u_{1}||_{V^{2}_{\sigma_{1}}}||P_{N_{2}}u_{2}||_{V^{2}_{\sigma_{2}}}||P_{N_{3}}u_{3}||_{V^{2}_{\sigma_{3}}}. 
\end{split}
\end{equation}
for some $\delta >0$.\\ 
{\rm (ii)} Let $d=1$, $2$, $3$, $s\geq 1$, $0<T<\infty$. 
For any $\sigma_{1}$, $\sigma_{2}$, $\sigma_{3}\in \R \backslash \{ 0\}$, any dyadic numbers $N_{1}$, $N_{2}$, $N_{3}\in 2^{\Z}$ and $1\leq j\leq 3$ 
, we have {\rm (\ref{hhh_sub})}.\\
{\rm (iii)} Let $s>s_{c}$, $0<T<\infty$ if $d=2$, $3$ and $s\geq 0$, $T=1$ if $d=1$. 
We assume $\sigma_{1}$, $\sigma_{2}$, $\sigma_{3}\in \R \backslash \{ 0\}$ satisfy $\sigma_{1}\sigma_{2}\sigma_{3}(1/\sigma_{1}+1/\sigma_{2}+1/\sigma_{3})>0$. 
For any dyadic numbers $N_{1}$, $N_{2}$, $N_{3}\in 2^{\Z}$ with $N_{1}\sim N_{2}\sim N_{3}$ and $1\leq j\leq 3$, we have {\rm (\ref{hhh_sub})}.
\end{prop}
\begin{proof}
By symmetry, it is enough to prove for $j=3$. 
We choose $\delta >0$ satisfying $\delta <(s-s_{c})/2$ and $\delta \ll 1$. 

First, we consider the case $d\geq 4$. By the H\"older's inequality and the Sobolev embedding $\dot{W}^{s_{c}+2\delta -1, 6d/(3d-4+12\delta )}(\R^{d})\hookrightarrow  L^{3d/4}(\R^{d})$, we have
\[
\begin{split}
&\left|\int_{0}^{T}\int_{\R^{d}}(P_{N_{1}}u_{1})(P_{N_{2}}u_{2})(P_{N_{3}}u_{3})dxdt\right|\\
&\lesssim  ||\ee_{[0,T)}||_{L^{1/\delta}_{t}}||P_{N_{1}}u||_{L^{3}_{t}L^{6d/(3d-4)}_{x}}||P_{N_{2}}u_{2}||_{L^{3}_{t}L^{6d/(3d-4)}_{x}}|||\nabla |^{s_{c}+2\delta -1}(P_{N_{3}}u_{3})||_{L^{p}_{t}L^{q}_{x}}
\end{split}
\]
with $(p,q)=(3/(1-3\delta ), 6d/(3d-4+12\delta )$ which is the admissible pair of the Strichartz estimate. 
Therefore we obtain (\ref{hhh_sub}) by (\ref{V_Stri}). 

Second, we consider the case $d=1$, $2$, $3$ and $\sigma_{1}$, $\sigma_{2}$, $\sigma_{3}\in \R \backslash \{ 0\}$ are arbitary. 
By the H\"older's inequality and (\ref{V_Stri}), we have
\[
\begin{split}
&\left|\int_{0}^{T}\int_{\R^{d}}(P_{N_{1}}u_{1})(P_{N_{2}}u_{2})(P_{N_{3}}u_{3})dxdt\right|\\
&\lesssim  ||\ee_{[0,T)}||_{L^{4/(4-d)}_{t}}||P_{N_{1}}u_{1}||_{L^{12/d}_{t}L^{3}_{x}}||P_{N_{2}}u_{2}||_{L^{12/d}_{t}L^{3}_{x}}||P_{N_{3}}u_{3}||_{L^{12/d}_{t}L^{3}_{x}}\\
&\lesssim 
T^{1-d/4}||P_{N_{1}}u_{1}||_{V^{2}_{\sigma_{1}}}||P_{N_{2}}u_{2}||_{V^{2}_{\sigma_{2}}}||P_{N_{3}}u_{3}||_{V^{2}_{\sigma_{3}}}
\end{split}
\]
and obtain (\ref{hhh_sub}) as $\delta =1-d/4$ for $s\geq 1$.

Third, we consider the case $d=2$, $3$ and $\sigma_{1}\sigma_{2}\sigma_{3}(1/\sigma_{1}+1/\sigma_{2}+1/\sigma_{3})>0$. 
In the proof of proposition~\ref{HHH_est}, for L.H.S of (\ref{HHH_1_est}), 
we use the H\"older's inequality and (\ref{highMproj}) with $p=2/(1-2\delta )$ instead of $p=2$. 
Then we have
\[
||Q_{\geq M}^{\sigma_{1}}f_{1,N_{1},T}||_{L^{2}_{tx}}
\leq ||\ee_{[0,T)}||_{L^{1/\delta}_{t}}||Q_{\geq M}^{\sigma_{1}}f_{1,N_{1},T}||_{L_{t}^{2/(1-2\delta )}L_{x}^{2}}
\lesssim T^{\delta}M^{-(1-2\delta )/2}||f_{1,N_{1},T}||_{V^{2}_{\sigma_{1}}}. 
\]
For the other part, by the same way of the proof of proposition~\ref{HHH_est}, we obtain (\ref{hhh_sub}).  

Finally, we consider the case $d=1$ and $\sigma_{1}\sigma_{2}\sigma_{3}(1/\sigma_{1}+1/\sigma_{2}+1/\sigma_{3})>0$. 
In the proof of proposition~\ref{HHH_est}, for L.H.S of (\ref{HHH_2_est}), 
we use (\ref{U8_est}) and $V_{-,rc}^{2}\hookrightarrow U^{8}$ instead of 
(\ref{U4_est}) and $V_{-,rc}^{2}\hookrightarrow U^{4}$. 
For the other part, by the same way of the proof of proposition~\ref{HHH_est}, we obtain (\ref{hhh_sub}) with $T=1$.  
\end{proof}
Proposition~\ref{HH_est_sub} and Proposition~\ref{HHH_est_sub} imply the following: 
\begin{cor}\label{HH-HL_est_sub}
Let $0<T<\infty$ if $d\geq 2$ and $T=1$ if $d=1$. We assume $\sigma_{1}$, $\sigma_{2}$, $\sigma_{3}\in \R \backslash \{ 0\}$ satisfy $(\sigma_{1}+\sigma_{2})(\sigma_{2}+\sigma_{3})(\sigma_{3}+\sigma_{1})\neq 0$.\\
{\rm (i)} Let $s>s_{c}$ if $d\geq 4$, and $s\geq 1$ if $d=1$, $2$, $3$. Then the estimate {\rm (\ref{hh_sub})} holds if we replace $\sum_{N_{3}\ll N_{2}}$ by $\sum_{N_{3}\lesssim N_{2}}$. \\
{\rm (ii)} Let $s>s_{c}$ if $d=2$, $3$ and $s\geq 0$ if $d=1$. We assume $\sigma_{1}$, $\sigma_{2}$, $\sigma_{3}\in \R \backslash \{ 0\}$ satisfy $\sigma_{1}\sigma_{2}\sigma_{3}(1/\sigma_{1}+1/\sigma_{2}+1/\sigma_{3})>0$. Then the estimate {\rm (\ref{hh_sub})} holds if we replace $\sum_{N_{3}\ll N_{2}}$ by $\sum_{N_{3}\lesssim N_{2}}$. 
\end{cor} 
Let $(i,j,k)$ is one of the permutation of $(1,2,3)$. 
If $\sigma_{i}+\sigma_{j}=0$, then Proposition~\ref{modul_est} (i) fails only for the case $|\xi_{k}|\ll |\xi_{i}|\sim |\xi_{j}|$. 
We obtain following estimates for the case $|\xi_{k}|\ll |\xi_{i}|\sim |\xi_{j}|$. 
\begin{cor}\label{HL_HH_est_sub2}
Let $s>s_{c}$ if $d\geq 4$, and $s>1$ if $d=2,3$.\\
{\rm (i)}\ We assume $\sigma_{1}$, $\sigma_{2}$, $\sigma_{3}\in \R \backslash \{ 0\}$ satisfy $\sigma_{2}+\sigma_{3}=0$ and $(\sigma_{1}+\sigma_{2})(\sigma_{3}+\sigma_{1})\neq 0$. 
Then for any $0<T< \infty$, and any dyadic numbers $N_{2}$, $N_{3}\in 2^{\Z}$ with $N_{2}\sim N_{3}$, we have
\begin{equation}\label{hl_sub_2}
\begin{split}
&\left|\sum_{N_{1}\ll N_{2}}N_{1}\int_{0}^{T}\int_{\R^{d}}(P_{N_{1}}u_{1})(P_{N_{2}}u_{2})(P_{N_{3}}u_{3})dxdt\right|\\
&\lesssim 
T^{\delta}\left(\displaystyle \sum_{N_{1}\ll N_{2}}(N_{1}\vee 1)^{2s}||P_{N_{1}}u_{1}||_{V^{2}_{\sigma_{1}}}^{2}\right)^{1/2}||P_{N_{2}}u_{2}||_{V^{2}_{\sigma_{2}}}||P_{N_{3}}u_{3}||_{V^{2}_{\sigma_{3}}}
\end{split}
\end{equation}
{\rm (ii)}\ We assume $\sigma_{1}$, $\sigma_{2}$, $\sigma_{3}\in \R \backslash \{ 0\}$ satisfy $\sigma_{1}+\sigma_{2}=0$ and $(\sigma_{2}+\sigma_{3})(\sigma_{3}+\sigma_{1})\neq 0$. 
Then for any $0<T< \infty$, and any dyadic numbers $N_{1}$, $N_{2}\in 2^{\Z}$ with $N_{1}\sim N_{2}$, we have
\begin{equation}\label{hh_sub_2}
\begin{split}
&\left(\sum_{N_{3}\lesssim N_{2}}N_{3}^{2s}
\sup_{||u_{3}||_{V^{2}_{\sigma_{3}}}=1}\left|N_{3}\int_{0}^{T}\int_{\R^{d}}(P_{N_{1}}u_{1})(P_{N_{2}}u_{2})(P_{N_{3}}u_{3})dxdt\right|^{2}
\right)^{1/2}\\
&\lesssim  T^{\delta}(N_{1}\vee 1)^{s}||P_{N_{1}}u_{1}||_{V^{2}_{\sigma_{1}}}(N_{2}\vee 1)^{s}||P_{N_{2}}u_{2}||_{V^{2}_{\sigma_{2}}}. 
\end{split}
\end{equation}
for some $\delta >0$. 
\end{cor}
\begin{proof}
By the H\"older's inequality, $V^{2}_{-,rc}\hookrightarrow L^{\infty}(\R ;L^{2})$ and (\ref{Vbilinear}), we have
\begin{equation}\label{hhh_sub_sub}
\begin{split}
&\left|N_{1}\int_{0}^{T}\int_{\R^{d}}(P_{N_{1}}u_{1})(P_{N_{2}}u_{2})(P_{N_{3}}u_{3})dxdt\right|\\
&\leq N_{1}||\ee_{[0,T)}||_{L^{2}_{t}}||(P_{N_{1}}u_{1})(P_{N_{2}}u_{2})||_{L^{2}_{tx}}||P_{N_{3}}u_{3}||_{L^{\infty}_{t}L^{2}_{x}}\\
&\lesssim T^{1/2}N_{1}^{s_{c}+1}\left(\frac{N_{1}}{N_{2}}\right)^{1/2}||P_{N_{1}}u_{1}||_{V^{2}_{\sigma_{1}}}||P_{N_{2}}u_{2}||_{V^{2}_{\sigma_{2}}}||P_{N_{3}}u_{3}||_{V^{2}_{\sigma_{3}}}
\end{split}
\end{equation}
for $N_{1}\ll N_{2}$. We use (\ref{hhh_sub_sub}) for the summation for $N_{1}<1$ 
and use (\ref{hhh_sub}) with $j=1$ for the summation for $1\leq N_{1}\ll N_{2}$. 
Then, we obtain (\ref{hl_sub_2}) by the Cauchy-Schwarz inequality for the dyadic sum. 

The estimate (\ref{hh_sub_2}) is obtained by using (\ref{hhh_sub}) with $j=3$. 
\end{proof}
%
%
\section{Proof of the well-posedness and the scattering \label{proof_thm}}\kuuhaku
In this section, we prove Theorems~\ref{wellposed_1}, ~\ref{wellposed_2}, ~\ref{ddqdnls_wp} and Corollary~\ref{sccat}. 
To begin with, we define the function spaces which spaces will be used to construct the solution. 
\begin{defn}\label{YZ_space}
Let $s$, $\sigma\in \R$.\\
{\rm (i)} We define $\dot{Z}^{s}_{\sigma}:=\{u\in C(\R ; \dot{H}^{s}(\R^{d}))\cap U^{2}_{\sigma }|\ ||u||_{\dot{Z}^{s}_{\sigma}}<\infty\}$ with the norm
\[
||u||_{\dot{Z}^{s}_{\sigma}}:=\left(\sum_{N}N^{2s}||P_{N}u||^{2}_{U^{2}_{\sigma}}\right)^{1/2}.
\]
{\rm (ii)} We define $Z^{s}_{\sigma}:=\{u\in C(\R ; H^{s}(\R^{d}))\cap U^{2}_{\sigma }|\ ||u||_{Z^{s}_{\sigma}}<\infty\}$ with the norm
\[
||u||_{Z^{s}_{\sigma}}:=||u||_{\dot{Z}^{0}_{\sigma}}+||u||_{\dot{Z}^{s}_{\sigma}}. 
\]
{\rm (iii)} We define $\dot{Y}^{s}_{\sigma}:=\{u\in C(\R ; \dot{H}^{s}(\R^{d}))\cap V^{2}_{-,rc,\sigma }|\ ||u||_{\dot{Y}^{s}_{\sigma}}<\infty\}$ with the norm
\[
||u||_{\dot{Y}^{s}_{\sigma}}:=\left(\sum_{N}N^{2s}||P_{N}u||^{2}_{V^{2}_{\sigma}}\right)^{1/2}.
\]
{\rm (iv)} We define $Y^{s}_{\sigma}:=\{u\in C(\R ; H^{s}(\R^{d}))\cap V^{2}_{-,rc,\sigma }|\ ||u||_{Y^{s}_{\sigma}}<\infty\}$ with the norm
\[
||u||_{Y^{s}_{\sigma}}:=||u||_{\dot{Y}^{0}_{\sigma}}+||u||_{\dot{Y}^{s}_{\sigma}}.
\]
\end{defn}
\begin{rem}
Let $E$ be a Banach space of continuous functions $f:\R\rightarrow H$, for some Hilbert space $H$. 
We also consider the corresponding restriction space to the interval $I\subset \R$ by
\[
E(I)=\{u\in C(I,H)|\exists v\in E\ s.t.\ v(t)=u(t),\ t\in I\}
\]
endowed with the norm $||u||_{E(I)}=\inf \{||v||_{E}|v(t)=u(t),\ t\in I\}$. 
Obviously, $E(I)$ is also a Banach space (see Remark\ 2.23 in \cite{HHK09}).
\end{rem}
We define the map $\Phi(u,v,w)=(\Phi_{T, \alpha, u_{0}}^{(1)}(w, v), \Phi_{T, \beta, v_{0}}^{(1)}(\overline{w}, v), \Phi_{T, \gamma, w_{0}}^{(2)}(u, \overline{v}))$ as
\[
\begin{split}
\Phi_{T, \sigma, \varphi}^{(1)}(f,g)(t)&:=e^{it\sigma \Delta}\varphi -I^{(1)}_{T,\sigma}(f,g)(t),\\
\Phi_{T, \sigma, \varphi}^{(2)}(f,g)(t)&:=e^{it\sigma \Delta}\varphi +I^{(2)}_{T,\sigma}(f,g)(t),
\end{split}
\] 
where
\[
\begin{split}
I^{(1)}_{T,\sigma}(f,g)(t)&:=\int_{0}^{t}\ee_{[0,T)}(t')e^{i(t-t')\sigma \Delta}(\nabla \cdot f(t'))g(t')dt',\\
I^{(2)}_{T,\sigma}(f,g)(t)&:=\int_{0}^{t}\ee_{[0,T)}(t')e^{i(t-t')\sigma \Delta}\nabla (f(t')\cdot g(t'))dt'.
\end{split}
\]
To prove the existence of the solution of (\ref{NLS_sys}), we prove that $\Phi$ is a contraction map 
on a closed subset of $\dot{Z}^{s}_{\alpha}([0,T])\times \dot{Z}^{s}_{\beta}([0,T])\times \dot{Z}^{s}_{\gamma}([0,T])$
 or $Z^{s}_{\alpha}([0,T])\times Z^{s}_{\beta}([0,T])\times Z^{s}_{\gamma}([0,T])$. 
 Key estimates are the followings:
\begin{prop}\label{Duam_est}
We assume that $\alpha$, $\beta$, $\gamma \in \R\backslash \{0\}$ satisfy the condition in {\rm Theorem~\ref{wellposed_1}}. Then for $s_{c}=d/2-1$ and any $0<T\leq \infty$, we have
\begin{align}
&||I_{T,\alpha}^{(1)}(w,v)||_{\dot{Z}^{s_{c}}_{\alpha}}\lesssim ||w||_{\dot{Y}^{s_{c}}_{\gamma}}||v||_{\dot{Y}^{s_{c}}_{\beta}},\label{Duam_al}\\
&||I_{T,\beta}^{(1)}(\overline{w},u)||_{\dot{Z}^{s_{c}}_{\beta}}\lesssim ||w||_{\dot{Y}^{s_{c}}_{\gamma}}||u||_{\dot{Y}^{s_{c}}_{\alpha}},\label{Duam_be}\\
&||I_{T,\gamma}^{(2)}(u,\overline{v})||_{\dot{Z}^{s_{c}}_{\gamma}}\lesssim ||u||_{\dot{Y}^{s_{c}}_{\alpha}}||v||_{\dot{Y}^{s_{c}}_{\beta}}.\label{Duam_ga}
\end{align}
\end{prop}
\begin{proof}
We prove only (\ref{Duam_ga}) since (\ref{Duam_al}) and (\ref{Duam_be}) are proved by the same way. 
We show the estimate
\begin{equation}\label{Nonlin_est}
||I_{T,\gamma}^{(2)}(u,\overline{v})||_{\dot{Z}^{s}_{\gamma}}\lesssim ||u||_{\dot{Y}^{s_{c}}_{\alpha}}||v||_{\dot{Y}^{s}_{\beta}}+||u||_{\dot{Y}^{s}_{\alpha}}||v||_{\dot{Y}^{s_{c}}_{\beta}}
\end{equation}
for $s\geq 0$. (\ref{Duam_ga}) follows from (\ref{Nonlin_est}) as $s=s_{c}$. We put $(u_{1},u_{2}):=(u,\overline{v})$ and $(\sigma_{1},\sigma_{2},\sigma_{3}):=(\alpha ,-\beta ,-\gamma )$. 
To obtain (\ref{Nonlin_est}), we use the argument of the proof of Theorem\ 3.2 in \cite{HHK09}. We define
\[
\begin{split}
J_{1}&:=\left|\left| \sum_{N_{2}}\sum_{N_{1}\ll N_{2}}I_{T,-\sigma_{3}}^{(2)}(P_{N_{1}}u_{1}, P_{N_{2}}u_{2})\right|\right|_{\dot{Z}^{s}_{-\sigma_{3}}},\\
J_{2}&:=\left|\left| \sum_{N_{2}}\sum_{N_{1}\sim N_{2}}I_{T,-\sigma_{3}}^{(2)}(P_{N_{1}}u_{1}, P_{N_{2}}u_{2})\right|\right|_{\dot{Z}^{s}_{-\sigma_{3}}},\\
J_{3}&:=\left|\left| \sum_{N_{1}}\sum_{N_{2}\ll N_{1}}I_{T,-\sigma_{3}}^{(2)}(P_{N_{1}}u_{1}, P_{N_{2}}u_{2})\right|\right|_{\dot{Z}^{s}_{-\sigma_{3}}},
\end{split}
\]
where implicit constants in $\ll$ actually depend on $\sigma_{1}$, $\sigma_{2}$, $\sigma_{3}$. 

First, we prove the estimate for $J_{1}$. By Theorem~\ref{duality}, we have
\[
\begin{split}
J_{1}&\leq \left\{ \sum_{N_{3}}N_{3}^{2s}\left(\sum_{N_{2}\sim N_{3}}\left|\left| e^{it\sigma_{3} \Delta}P_{N_{3}}\sum_{N_{1}\ll N_{2}}I_{T,-\sigma_{3}}^{(2)}(P_{N_{1}}u_{1}, P_{N_{2}}u_{2})\right|\right|_{U^{2}}\right)^{2}\right\}^{1/2}\\
&=\left\{\sum_{N_{3}}N_{3}^{2s}
\left( \sum_{N_{2}\sim N_{3}}\sup_{||u_{3}||_{V^{2}_{\sigma_{3}}}=1}\left|\sum_{N_{1}\ll N_{2}}N_{3}\int_{0}^{T}\int_{\R^{d}}(P_{N_{1}}u_{1})(P_{N_{2}}u_{2})(P_{N_{3}}u_{3})dxdt\right|\right)^{2}\right\}^{1/2}. 
\end{split}
\]
Therefore by Proposition~\ref{HL_est}, we have
\[
\begin{split}
J_{1}&\lesssim 
\left\{\sum_{N_{3}}N_{3}^{2s}
\left(\sum_{N_{2}\sim N_{3}}\sup_{||u_{3}||_{V^{2}_{\sigma_{3}}}=1}\left(\sum_{N_{1}\ll N_{2}}N_{1}^{2s_{c}}||P_{N_{1}}u_{1}||_{V^{2}_{\sigma_{1}}}^{2}\right)^{1/2}||P_{N_{2}}u_{2}||_{V^{2}_{\sigma_{2}}}||P_{N_{3}}u_{3}||_{V^{2}_{\sigma_{3}}}\right)^{2} 
\right\}^{1/2}\\
&\lesssim 
\left(\sum_{N_{1}}N_{1}^{2s_{c}}||P_{N_{1}}u_{1}||_{V^{2}_{\sigma_{1}}}^{2}\right)^{1/2}
\left(\sum_{N_{2}}N_{2}^{2s}||P_{N_{2}}u_{2}||_{V^{2}_{\sigma_{2}}}^{2}\right)^{1/2}\\
&= ||u_{1}||_{\dot{Y}^{s_{c}}_{\sigma_{1}}}||u_{2}||_{\dot{Y}^{s}_{\sigma_{2}}}.
\end{split}
\]

Second, we prove the estimate for $J_{2}$. By Theorem~\ref{duality}, we have
\[
\begin{split}
J_{2}&\leq
\sum_{N_{2}}\sum_{N_{1}\sim N_{2}}\left(\sum_{N_{3}\lesssim N_{2}}N_{3}^{2s}||e^{it\sigma_{3} \Delta}P_{N_{3}}I_{T,-\sigma_{3}}^{(2)}(P_{N_{1}}u_{1}, P_{N_{2}}u_{2})||_{U^{2}}^{2}\right)^{1/2}\\
&=\sum_{N_{2}}\sum_{N_{1}\sim N_{2}}\left(\sum_{N_{3}\lesssim N_{2}}N_{3}^{2s}
\sup_{||u_{3}||_{V^{2}_{\sigma_{3}}}=1}\left| \int_{0}^{T}\int_{\R^{d}}(P_{N_{1}}u_{1})(P_{N_{2}}u_{2})(P_{N_{3}}u_{3})dxdt\right|^{2}\right)^{1/2}.
\end{split}
\]
Therefore by {\rm Corollary~\ref{HH-HL_est}} and Cauchy-Schwarz inequality for dyadic sum, we have
\[
\begin{split}
J_{2}&\lesssim
\sum_{N_{2}}\sum_{N_{1}\sim N_{2}}N_{1}^{s_{c}}||P_{N_{1}}u_{1}||_{V^{2}_{\sigma_{1}}}N_{2}^{s}||P_{N_{2}}u_{2}||_{V^{2}_{\sigma_{2}}}\\
&\lesssim \left(\sum_{N_{1}}N_{1}^{2s_{c}}||P_{N_{1}}u_{1}||_{V^{2}_{\sigma_{1}}}^{2}\right)^{1/2}
\left(\sum_{N_{2}}N_{2}^{2s}||P_{N_{2}}u_{2}||_{V^{2}_{\sigma_{2}}}^{2}\right)^{1/2}\\
&= ||u_{1}||_{\dot{Y}^{s_{c}}_{\sigma_{1}}}||u_{2}||_{\dot{Y}^{s}_{\sigma_{2}}}.
\end{split}
\]

Finally, we prove the estimate for $J_{3}$. By the same manner as for $J_{1}$, we have
\[
J_{3}\lesssim ||u_{1}||_{\dot{Y}^{s}_{\sigma_{1}}}||u_{2}||_{\dot{Y}^{s_{c}}_{\sigma_{2}}}. 
\]

Therefore, we obtain (\ref{Nonlin_est}) since $||u_{1}||_{\dot{Y}^{s}_{\sigma_{1}}}=||u||_{\dot{Y}^{s}_{\alpha}}$ and $||u_{2}||_{\dot{Y}^{s_{c}}_{\sigma_{2}}}=||v||_{\dot{Y}^{s_{c}}_{\beta}}$. 
\end{proof}
\begin{cor}\label{Duam_est_inhom}
We assume that $\alpha$, $\beta$, $\gamma \in \R\backslash \{0\}$ satisfy the condition in {\rm Theorem~\ref{wellposed_1}}. 
Then for $s\geq s_{c}$ $(=d/2-1)$ and any $0<T\leq \infty$, we have
\begin{align}
&||I_{T,\alpha}^{(1)}(w,v)||_{Z^{s}_{\alpha}}\lesssim ||w||_{Y^{s}_{\gamma}}||v||_{Y^{s}_{\beta}},\label{Duam_al_inhom}\\
&||I_{T,\beta}^{(1)}(\overline{w},u)||_{Z^{s}_{\beta}}\lesssim ||w||_{Y^{s}_{\gamma}}||u||_{Y^{s}_{\alpha}},\label{Duam_be_inhom}\\
&||I_{T,\gamma}^{(2)}(u,\overline{v})||_{Z^{s}_{\gamma}}\lesssim ||u||_{Y^{s}_{\alpha}}||v||_{Y^{s}_{\beta}}.\label{Duam_ga_inhom}
\end{align}
\end{cor}
\begin{proof}
We prove only (\ref{Duam_ga_inhom}) since (\ref{Duam_al_inhom}) and (\ref{Duam_be_inhom}) are proved by the same way. 
By (\ref{Nonlin_est}), we have
\[
\begin{split}
||I_{T,\gamma}^{(2)}(u,\overline{v})||_{Z^{s}_{\gamma}}
&=||I_{T,\gamma}^{(2)}(u,\overline{v})||_{\dot{Z}^{0}_{\gamma}}+||I_{T,\gamma}^{(2)}(u,\overline{v})||_{\dot{Z}^{s}_{\gamma}}\\
&\lesssim ||u||_{\dot{Y}^{s_{c}}_{\alpha}}||v||_{\dot{Y}^{0}_{\beta}}+||u||_{\dot{Y}^{0}_{\alpha}}||v||_{\dot{Y}^{s_{c}}_{\beta}}+||u||_{\dot{Y}^{s_{c}}_{\alpha}}||v||_{\dot{Y}^{s}_{\beta}}+||u||_{\dot{Y}^{s}_{\alpha}}||v||_{\dot{Y}^{s_{c}}_{\beta}}.
\end{split}
\]
We decompose $u=P_{0}u+(Id-P_{0})u$ and $v=P_{0}v+(Id-P_{0})v$. Since
\[
\begin{split}
&||P_{0}u||_{\dot{Y}^{s_{c}}_{\alpha}}\lesssim ||P_{0}u||_{\dot{Y}^{0}_{\alpha}},\ ||(Id-P_{0})u||_{\dot{Y}^{s_{c}}_{\alpha}}\lesssim ||(Id-P_{0})u||_{\dot{Y}^{s}_{\alpha}},\\
&||P_{0}v||_{\dot{Y}^{s_{c}}_{\beta}}\lesssim ||P_{0}v||_{\dot{Y}^{0}_{\beta}},\ ||(Id-P_{0})v||_{\dot{Y}^{s_{c}}_{\beta}}\lesssim ||(Id-P_{0})v||_{\dot{Y}^{s}_{\beta}}
\end{split}
\]
for $s\geq s_{c}$, we obtain (\ref{Duam_ga_inhom}). 
\end{proof}
\begin{prop}\label{Duam_est_inhom_sub}\kuuhaku \\
{\rm (i)}\ Let $d\geq 2$. We assume that $\alpha$, $\beta$, $\gamma \in \R\backslash \{0\}$ and $s\in \R$ 
satisfy the condition in {\rm Theorem~\ref{wellposed_2}}. 
Then there exists $\delta >0$ such that for any $0<T< \infty$, we have
\begin{align}
&||I_{T,\alpha}^{(1)}(w,v)||_{Z^{s}_{\alpha}}\lesssim T^{\delta}||w||_{Z^{s}_{\gamma}}||v||_{Z^{s}_{\beta}},\label{Duam_al_inhom_sub}\\
&||I_{T,\beta}^{(1)}(\overline{w},u)||_{Z^{s}_{\beta}}\lesssim T^{\delta}||w||_{Z^{s}_{\gamma}}||u||_{Z^{s}_{\alpha}},\label{Duam_be_inhom_sub}\\
&||I_{T,\gamma}^{(2)}(u,\overline{v})||_{Z^{s}_{\gamma}}\lesssim T^{\delta}||u||_{Z^{s}_{\alpha}}||v||_{Z^{s}_{\beta}}.\label{Duam_ga_inhom_sub}
\end{align}
{\rm (ii)}\ Let $d=1$. We assume that $\alpha$, $\beta$, $\gamma \in \R\backslash \{0\}$ and $s\in \R$ 
satisfy the condition in {\rm Theorem~\ref{wellposed_2}} except the case $1>s\geq 1/2$, $\theta <0$ and $(\alpha -\gamma )(\beta +\gamma ) \neq 0$. Then we have {\rm (\ref{Duam_al_inhom_sub})--(\ref{Duam_ga_inhom_sub})} with $T=1$. 
\end{prop}
\begin{proof}
We obtain (\ref{Duam_al_inhom_sub})--(\ref{Duam_ga_inhom_sub}) by using Proposition~\ref{HL_est_sub} and Corollary~\ref{HH-HL_est_sub} 
if $\alpha \neq \beta$, using Corollary~\ref{HL_HH_est_sub2} if $d\geq 2$ and $\alpha =\beta$ instead of Proposition~\ref{HL_est} and Corollary~\ref{HH-HL_est} in the proof of Proposition~\ref{Duam_est}. 
\end{proof}
\begin{proof}[\rm{\bf{Proof of Theorem~\ref{wellposed_1}.}}]
We prove only the homogeneous case. The inhomogeneous case is also proved by the same way. 
For $s\in \R$ and interval $I\subset \R$, we define
\begin{equation}\label{Xs_norm}
\dot{X}^{s}(I):=\dot{Z}^{s}_{\alpha}(I)\times \dot{Z}^{s}_{\beta}(I)\times \dot{Z}^{s}_{\gamma}(I). 
\end{equation}
Furthermore for $r>0$, we define 
\begin{equation}\label{Xrs_norm}
\dot{X}^{s}_{r}(I)
:=\left\{(u,v,w)\in \dot{X}^{s}(I)\left|\ ||u||_{\dot{Z}_{\alpha}^{s}(I)},\ ||v||_{\dot{Z}_{\beta}^{s}(I)},\ ||w||_{\dot{Z}_{\gamma}^{s}(I)}\leq 2r \right.\right\}
\end{equation}
which is a closed subset of $\dot{X}^{s}(I)$. 
Let $(u_{0}$, $v_{0}$, $w_{0})\in B_{r}(\dot{H}^{s_{c}}\times \dot{H}^{s_{c}}\times \dot{H}^{s_{c}})$ be given. For $(u,v,w)\in \dot{X}^{s_{c}}_{r}([0,\infty ))$, 
we have
\[
\begin{split}
||\Phi^{(1)}_{T,\alpha ,u_{0}}(w, v)||_{\dot{Z}_{\alpha}^{s_{c}}([0,\infty ))}&\leq ||u_{0}||_{\dot{H}^{s_{c}}} +C||w||_{\dot{Z}_{\gamma}^{s_{c}}([0,\infty ))}||v||_{\dot{Z}_{\beta}^{s_{c}}([0,\infty ))}\leq r(1+4Cr),\\
||\Phi^{(1)}_{T,\beta ,v_{0}}(\overline{w}, u)||_{\dot{Z}_{\beta}^{s_{c}}([0,\infty ))}&\leq ||v_{0}||_{\dot{H}^{s_{c}}} +C||w||_{\dot{Z}_{\gamma}^{s_{c}}([0,\infty ))}||u||_{\dot{Z}_{\alpha}^{s_{c}}([0,\infty ))}\leq r(1+4Cr),\\
||\Phi^{(2)}_{T,\gamma ,w_{0}}(u, \overline{v})||_{\dot{Z}_{\gamma}^{s_{c}}([0,\infty ))}&\leq ||w_{0}||_{\dot{H}^{s_{c}}} +C||u||_{\dot{Z}_{\alpha}^{s_{c}}([0,\infty ))}||v||_{\dot{Z}_{\beta}^{s_{c}}([0,\infty ))}\leq r(1+4Cr)
\end{split}
\]
and
\[
\begin{split}
||\Phi^{(1)}_{T,\alpha ,u_{0}}(w_{1}, v_{1})-\Phi^{(1)}_{T,\alpha ,u_{0}}(w_{2}, v_{2})||_{\dot{Z}_{\alpha}^{s_{c}}([0,\infty ))}
&\leq 2Cr\left( ||w_{1}-w_{2}||_{\dot{Z}_{\gamma}^{s_{c}}([0,\infty ))}+||v_{1}-v_{2}||_{\dot{Z}_{\beta}^{s_{c}}([0,\infty ))}\right),\\
||\Phi^{(1)}_{T,\beta ,v_{0}}(\overline{w_{1}}, u_{1})-\Phi^{(1)}_{T,\beta ,v_{0}}(\overline{w_{2}}, u_{2})||_{\dot{Z}_{\beta}^{s_{c}}([0,\infty ))}
&\leq 2Cr\left( ||w_{1}-w_{2}||_{\dot{Z}_{\gamma}^{s_{c}}([0,\infty ))}+||u_{1}-u_{2}||_{\dot{Z}_{\alpha}^{s_{c}}([0,\infty ))}\right),\\
||\Phi^{(2)}_{T,\gamma ,w_{0}}(u_{1}, \overline{v_{1}})-\Phi^{(1)}_{T,\gamma ,w_{0}}(u_{2}, \overline{v_{2}})||_{\dot{Z}_{\gamma}^{s_{c}}([0,\infty ))}
&\leq 2Cr\left( ||u_{1}-u_{2}||_{\dot{Z}_{\alpha}^{s_{c}}([0,\infty ))}+||v_{1}-v_{2}||_{\dot{Z}_{\beta}^{s_{c}}([0,\infty ))}\right)
\end{split}
\]
by Proposition~\ref{Duam_est} and
\[
||e^{i\sigma t\Delta}\varphi ||_{\dot{Z}^{s_{c}}_{\sigma}([0,\infty ))}\leq ||\ee_{[0,\infty )}e^{i\sigma t\Delta}\varphi ||_{\dot{Z}^{s_{c}}_{\sigma}}\leq ||\varphi ||_{\dot{H}^{s_{c}}}, 
\] 
where $C$ is an implicit constant in (\ref{Duam_al})--(\ref{Duam_ga}). Therefore if we choose $r$ satisfying
\[
r <(4C)^{-1},
\]
then $\Phi$ is a contraction map on $\dot{X}^{s_{c}}_{r}([0,\infty ))$. 
This implies the existence of the solution of the system (\ref{NLS_sys}) and the uniqueness in the ball $\dot{X}^{s_{c}}_{r}([0,\infty ))$. 
The Lipschitz continuously of the flow map is also proved by similar argument. 
\end{proof} 
Theorem~\ref{wellposed_2} except the case $d=1$, $1>s\geq 1/2$, $\theta <0$ and $(\alpha -\gamma )(\beta +\gamma ) \neq 0$ is proved by the same way for the proof of Theorem~\ref{wellposed_1}. 
\begin{rem}
For $d=1$ and $s>s_{c}$ (in particular $s\geq 0$), we can assume the $H^{s}$-norm of the initial data is small enough 
by the scaling {\rm (\ref{scaling_tr})} with large $\lambda$ since $s_{c}<0$. 
\end{rem}
\begin{proof}[\rm{\bf{Proof of Corollary~\ref{sccat}.}}]
We prove only the homogeneous case. The inhomogeneous case is also proved by the same way. 
By Proposition~\ref{Duam_est}, 
the global solution $(u,v,w)\in \dot{X}^{s_{c}}([0,\infty ))$ of (\ref{NLS_sys}) which was constructed in Theorem~\ref{wellposed_1} 
satisfies
\[
N^{s_{c}}(e^{-it\alpha \Delta}P_{N}I_{\infty , \alpha}^{(1)}(w,v),\ e^{-it\beta \Delta}P_{N}I_{\infty , \beta}(\overline{w},u),\ e^{-it\gamma \Delta}P_{N}I_{\infty,\gamma}^{(2)}(u,\overline{v}))\in V^{2}\times V^{2}\times V^{2}
\]
for each $N\in 2^{\Z}$. This implies that
\[
(u_{+}, v_{+}, w_{+}):=\lim_{t\rightarrow \infty}(u_{0}-e^{-it\alpha \Delta}I_{\infty , \alpha}^{(1)}(w,v),\ v_{0}-e^{-it\beta \Delta}I_{\infty , \beta}(\overline{w},u),\ w_{0}+e^{-it\gamma \Delta}I_{\infty,\gamma}^{(2)}(u,\overline{v}))
\]
exists in $\dot{H}^{s_{c}}\times \dot{H}^{s_{c}}\times \dot{H}^{s_{c}}$ by Proposition~\ref{upvpprop}\ {\rm (4)}. 
Then we obtain
\[
(u,v,w)-(e^{it\alpha \Delta}u_{+}, e^{it\beta \Delta}v_{+}, e^{it\gamma \Delta}w_{+})\rightarrow 0
\]
in $\dot{H}^{s_{c}}\times \dot{H}^{s_{c}}\times \dot{H}^{s_{c}}$ as $t\rightarrow \infty$. 
\end{proof}
Theorem~\ref{ddqdnls_wp} is proved by using the estimate (\ref{Duam_al}) and (\ref{Duam_al_inhom}) for 
$(\alpha, \beta, \gamma ) =(-1,1,1)$.  
%
%
\section{A priori estimates \label{proof_apriori}}
In this section, we prove Theorem~\ref{global_extend}. We define
\[
\begin{split}
M(u,v,w)&:=2||u||_{L^{2}_{x}}^{2}+||v||_{L^{2}_{x}}^{2}+||w||_{L^{2}_{x}}^{2}\\
H(u,v,w)&:=\alpha ||\nabla u||_{L^{2}_{x}}^{2}+\beta ||\nabla v||_{L^{2}_{x}}^{2}+\gamma ||\nabla w ||_{L^{2}_{x}}^{2}+2{\rm Re}(w ,\nabla (u\cdot \overline{v}))_{L^{2}_{x}}
\end{split}
\]
and put $M_{0}:=M(u_{0}, v_{0}, w_{0})$, $H_{0}:=H(u_{0}, v_{0}, w_{0})$.
\begin{prop}\label{conservation}
For the smooth solution $(u, v, w)$ of the system\ {\rm{(\ref{NLS_sys})}}, we have
\[
M(u,v,w)=M_{0},\ H(u,v,w)=H_{0}
\]
\end{prop}
\begin{proof}
For the system
\begin{align}
&(i\partial_{t}+\alpha \Delta )u=-(\nabla \cdot w )v\label{ssu}\\
&(i\partial_{t}+\beta \Delta )v=-(\nabla \cdot \overline{w})u\label{ssv}\\
&(i\partial_{t}+\gamma \Delta )w =\nabla (u\cdot \overline{w}) \label{ssw}, 
\end{align}
We have the conservation law for $M$ by calculating 
\[
{\rm Im}\int_{\R^{d}}\{ (-2u\times \overline{{\rm (\ref{ssu})})}+(\overline{v}\times {\rm (\ref{ssv})})+(\overline{w} \times {\rm (\ref{ssw})})\}dx
\]
and for $H$ by calculating
\[
{\rm Re}\int_{\R^{d}}\{ (\partial_{t}u\times \overline{{\rm (\ref{ssu})}})+(\partial_{t}\overline{v}\times {\rm (\ref{ssv})})+(\partial_{t}\overline{w} \times {\rm (\ref{ssw})})\}dx .
\]
\end{proof}
The following a priori estimates imply Theorem~\ref{global_extend}. 
\begin{prop}\label{apriori_est}
We assume $\alpha$, $\beta$ and $\gamma$ have the same sign and put 
\[
\rho_{max}:= \max \{|\alpha |, |\beta |, |\gamma |\},\ \rho_{min}:=\min \{|\alpha |, |\beta |, |\gamma |\} . 
\]
{\rm (i)} Let $d=1$, $2$. For the data $(u_{0}, v_{0}, w_{0})\in H^{1}\times H^{1}\times H^{1}$ satisfying
\begin{equation}\label{apriori_condi_1}
M_{0}^{1-d/4}\ll \rho_{min}, 
\end{equation}
there exists $C>0$ such that for the solution $(u,v,w)\in \left(C([0,T];H^{1})\right)^{3}$ of {\rm{(\ref{NLS_sys})}}, the following estimate holds: 
\begin{equation}\label{apriori_1}
\sup_{0\leq t\leq T}\left( ||\nabla u(t)||_{L^{2}_{x}}^{2}+||\nabla v(t)||_{L^{2}_{x}}^{2}+||\nabla w(t) ||_{L^{2}_{x}}^{2}\right)\leq \frac{H_{0}+CM_{0}^{1-d/4}}{\rho_{min}-CM_{0}^{1-d/4}}. 
\end{equation} 
{\rm (ii)} Let $d=3$. If the data $(u_{0}, v_{0}, w_{0})\in H^{1}\times H^{1}\times H^{1}$ satisfies 
\begin{equation}\label{apriori_condi_2}
||\nabla u_{0}||_{L^{2}_{x}}^{2}+||\nabla v_{0}||_{L^{2}_{x}}^{2}+||\nabla w_{0}||_{L^{2}_{x}}^{2}
< \epsilon^{2}/\rho_{max}
\end{equation}
for some $\epsilon$ with $0<\epsilon \ll 1$, then for the solution $(u,v,w)\in \left(C([0,T];H^{1})\right)^{3}$ of {\rm{(\ref{NLS_sys})}}, the following estimate holds: 
\begin{equation}\label{apriori_2}
\sup_{0\leq t\leq T}\left(||\nabla u(t)||_{L^{2}_{x}}^{2}+||\nabla v(t)||_{L^{2}_{x}}^{2}+||\nabla w(t) ||_{L^{2}_{x}}^{2}\right)<3\epsilon^{2}/\rho_{\min}.
\end{equation}
\end{prop}
\begin{proof}
We put
\[
F=F(t):=||\nabla u(t)||_{L^{2}_{x}}^{2}+||\nabla v(t)||_{L^{2}_{x}}^{2}+||\nabla w(t)||_{L^{2}_{x}}^{2}.
\]
Since $\alpha$, $\beta$ and $\gamma$ are same sign, we have
\[
F\leq \frac{1}{\rho_{min}}(H(u,v,w)+2|( (\nabla \cdot w), (u\cdot \overline{v}))_{L^{2}_{x}}|). 
\]
By the Cauchy-Schwarz inequality and the Gagliardo-Nirenberg inequality we have
\[
\begin{split}
|( (\nabla \cdot w), (u\cdot \overline{v}))_{L^{2}_{x}}|)
&\leq ||\nabla \cdot w ||_{L^{2}_{x}}||u||_{L^{4}_{x}}||v||_{L^{4}_{x}}\\
&\lesssim ||\nabla \cdot w||_{L^{2}_{x}}||u||_{L^{2}_{x}}^{1-d/4}||\nabla u||_{L^{2}_{x}}^{d/4}||v||_{L^{2}_{x}}^{1-d/4}||\nabla v||_{L^{2}_{x}}^{d/4}\\
&\lesssim M(u,v,w)^{1-d/4}F^{(d+2)/4}
\end{split}
\]
for $d\leq 4$. Therefore, by using Proposition~\ref{conservation}, we obtain
\begin{equation}\label{F_est}
F\leq \frac{1}{\rho_{min}}\left(H_{0}+CM_{0}^{1-d/4}F^{(d+2)/4}\right)
\end{equation}
for some constant $C>0$. For $d\leq 2$ we have $F^{(d+2)/4}\leq 1+F$ 
because of $(d+2)/4\leq 1$. 
Therefore if (\ref{apriori_condi_1}) holds, then the estimate {\rm (\ref{apriori_1})} follows from {\rm (\ref{F_est})}. 

By the same argument as above, we obtain
\[
\begin{split}
H_{0}&\leq \rho_{max}F(0)+2|((\nabla \cdot w(0)), (u(0)\cdot \overline{v(0)}))_{L^{2}_{x}}|
\leq \rho_{max}F(0)+CM_{0}^{1-d/4}F(0)^{(d+2)/4}
\end{split}
\]
for some constant $C>0$ and $d\leq 4$. Therefore if (\ref{apriori_condi_2}) holds
for some $\epsilon$ with $0<\epsilon \ll 1$, we have
\[
H_{0}<\epsilon^{2}(1+CM_{0}^{1-d/4}\rho_{max}^{-(d+2)/4}\epsilon^{(d-2)/2}). 
\]
By choosing $\epsilon$ sufficiently small, we have 
$H_{0}<2\epsilon^{2}$
for $d=3$\ (and also $d=4$). Therefore the estimate 
\begin{equation}\label{F_est_2}
F\leq \frac{1}{\rho_{min}}\left(2\epsilon^{2}+CM_{0}^{1-d/4}F^{(d+2)/4}\right)
\end{equation}
follows from {\rm (\ref{F_est})}. If there exists $t_{0}\in [0, T]$ such that 
$F(t_{0})< 4\epsilon^{2}/\rho_{min}$ 
for sufficiently small $\epsilon$, then we have 
$F(t_{0})<3\epsilon^{2}/\rho_{min}$
by {\rm (\ref{F_est_2})}. 
Since $F(0)<\epsilon^{2}/\rho_{min}<4\epsilon^{2}/\rho_{min}$ and $F(t)$ is continuous with respect to $t$, 
we obtain {\rm (\ref{apriori_2})}. 
\end{proof}


\section{$C^{2}$-ill-posedness \label{proof_ill}}
In this section, we prove Theorem~\ref{notC2}. 
We rewrite Theorem~\ref{notC2} as follows: 
\begin{thm}Let $d\geq 1$, $0<T\ll 1$ and $\alpha$, $\beta$, $\gamma \in \R\backslash \{0\}$. 
We assume $s\in \R$ if $(\alpha -\gamma )(\beta +\gamma )=0$, $s<1$ if $\alpha \beta \gamma (1/\alpha -1/\beta -1/\gamma) =0$, 
and $s<1/2$ if $\alpha \beta \gamma (1/\alpha -1/\beta -1/\gamma) <0$. 
Then for every $C>0$ there exist $f$, $g\in H^{s}(\R^{d})$ such that
\begin{equation}\label{flow_map_estimate}
\sup_{0\leq t\leq T}\left|\left|\int_{0}^{t}e^{i(t-t')\gamma \Delta}\nabla ((e^{it'\alpha \Delta}f)(\overline{e^{it'\beta \Delta}g}))dt'\right|\right|_{H^{s}}
\geq C||f||_{H^{s}}||g||_{H^{s}}. 
\end{equation}
\end{thm}
\begin{proof}
We prove only for $d=1$. 
For $d\geq 2$, it is enough to replace $D_{1}$, $D_{2}$ and $D$ by $D_{1}\times [0,1]^{d-1}$, $D_{2}\times [0,1]^{d-1}$ and $D\times [1/2,1]^{d-1}$ 
in the following argument. 
We use the argument of the proof of Theorem 1 in \cite{MST01}.  
For the sets $D_{1}$, $D_{2}\subset \R$, we define the functions $f$, $g\in H^{s}(\R)$ as
\[
\widehat{f}(\xi )=\ee_{D_{1}}(\xi ),\ \widehat{g}(\xi )=\ee_{D_{2}}(\xi ).  
\]

First, we consider the case $(\alpha -\gamma )(\beta +\gamma )=0$. 
We assume $\alpha -\gamma =0$. (For the case $\beta +\gamma =0$ is proved by similar argument. ) 
We put $M:=-(\beta +\gamma)/2\gamma$, 
then we have
\[
\alpha |\xi_{1}|^{2}-\beta |\xi -\xi_{1}|^{2}-\gamma |\xi |^{2}
=2\gamma \{\xi_{1}-M(\xi -\xi_{1})\}(\xi -\xi_{1}). 
\]
For $N\gg 1$, we define the sets $D_{1}$, $D_{2}$ and $D\subset \R$ as
\[
D_{1}:=[N,\ N+N^{-1}],\ D_{2}:=[N^{-1},\ 2N^{-1}],\ D:= [N+3N^{-1}/2,\ N+2N^{-1}]
\]
Then, we have 
\[
||f||_{H^{s}}\sim N^{s-1/2},\ ||g||_{H^{s}}\sim N^{-1/2} ,\ |(\widehat{f}*\widehat{g})(\xi) |\gtrsim N^{-1}\ee_{D}(\xi ) 
\]
and
\[
\int_{0}^{t}e^{-it'(\alpha |\xi_{1}|^{2}-\beta |\xi -\xi_{1}|^{2}-\gamma |\xi |^{2})}
dt'\sim t
\]
for any $\xi \in D_{1}$ satisfying $\xi -\xi_{1}\in D_{2}$ and $0\leq t\ll 1$.   
This implies
\[
\sup_{0\leq t\leq T}\left|\left|\int_{0}^{t}e^{i(t-t')\gamma \Delta}\nabla ((e^{it'\alpha \Delta}f)(\overline{e^{it'\beta \Delta}g}))dt'\right|\right|_{H^{s}}
\gtrsim  N^{s-1/2} 
\]
Therefore we obtain {\rm (\ref{flow_map_estimate})} because $s-1/2>s-1$ for any $s\in \R$. 

Second, we consider the case $\alpha \beta \gamma (1/\alpha -1/\beta -1/\gamma)= 0$. 
We put $M:=\gamma /(\alpha -\gamma)$, 
then $M\neq -1$ since $\alpha \neq 0$ and we have
\[
\alpha |\xi_{1}|^{2}-\beta |\xi -\xi_{1}|^{2}-\gamma |\xi |^{2}
=(\alpha -\gamma )|\xi_{1}-M(\xi -\xi_{1})|^{2}. 
\]
For $N\gg 1$, we define the sets $D_{1}$, $D_{2}$ and $D\subset \R$ as
\[
D_{1}:=[N,\ N+1],\ D_{2}:=[N/M,\ N/M+1/|M|],\ D:=[(1+1/M)N+1/2,\ (1+1/M)N+1]. 
\]
Then, we have 
\[
||f||_{H^{s}}\sim N^{s},\ ||g||_{H^{s}}\sim N^{s} ,\ |(\widehat{f}*\widehat{g})(\xi) |\gtrsim \ee_{D}(\xi )
\]
and
\[
\int_{0}^{t}e^{-it'(\alpha |\xi_{1}|^{2}-\beta |\xi -\xi_{1}|^{2}-\gamma |\xi |^{2})}
dt'\sim t
\]
for any $\xi \in D_{1}$ satisfying $\xi -\xi_{1}\in D_{2}$ and $0\leq t\ll 1$. This implies
\[
\sup_{0\leq t\leq T}\left|\left|\int_{0}^{t}e^{i(t-t')\gamma \Delta}\nabla ((e^{it'\alpha \Delta}f)(\overline{e^{it'\beta \Delta}g}))dt'\right|\right|_{H^{s}}
\gtrsim  N^{s+1}
\]
Therefore we obtain {\rm (\ref{flow_map_estimate})} because $s+1>2s$ for any $s<1$. 

Finally, we consider the case $(\alpha -\gamma )(\beta +\gamma )\neq 0$ and $\alpha \beta \gamma (1/\alpha -1/\beta -1/\gamma)< 0$. We put
\[
M_{\pm}:=\frac{\gamma}{\alpha -\gamma}\pm \frac{1}{\alpha -\gamma }\sqrt{-\alpha \beta \gamma \left(\frac{1}{\alpha}-\frac{1}{\beta}-\frac{1}{\gamma}\right)}, 
\]
then $M_{\pm}\in \R$ and $M_{+}\neq M_{-}$ since $\alpha \beta \gamma (1/\alpha -1/\beta -1/\gamma)< 0$, and we have 
\[
\alpha |\xi_{1}|^{2}-\beta |\xi -\xi_{1}|^{2}-\gamma |\xi |^{2}
=(\alpha +\gamma )\{\xi_{1}-M_{+}(\xi -\xi_{1})\} \{\xi_{1}-M_{-}(\xi -\xi_{1})\}. 
\]
Because $M_{+}\neq M_{-}$, at least one of $M_{+}$ and $M_{-}$ is not equal to $-1$. 
We can assume $M_{+}\neq -1$ without loss of generality. 
For $N\gg 1$, we define the sets $D_{1}$, $D_{2}$ and $D\subset \R$ as
\[
\begin{split}
&D_{1}:=[N,\ N+N^{-1}],\ 
D_{2}:=[N/M_{+},\ N/M_{+}+N^{-1}/|M_{+}|],\\ 
&D:=[(1+1/M_{+})N+N^{-1}/2,\ (1+1/M_{+})N+N^{-1}]. 
\end{split}
\]
Then, we have 
\[
||f||_{H^{s}}\sim N^{s-1/2},\ ||g||_{H^{s}}\sim N^{s-1/2} ,\ |(\widehat{f}*\widehat{g})(\xi) |\gtrsim N^{-1}\ee_{D}(\xi )
\]
and
\[
\int_{0}^{t}e^{-it'(\alpha |\xi_{1}|^{2}-\beta |\xi -\xi_{1}|^{2}-\gamma |\xi |^{2})}
dt'\sim t
\]
for any $\xi \in D_{1}$ satisfying $\xi -\xi_{1}\in D_{2}$ and $0\leq t\ll 1$. This implies
\[
\sup_{0\leq t\leq T}\left|\left|\int_{0}^{t}e^{i(t-t')\gamma \Delta}\nabla ((e^{it'\alpha \Delta}f)(\overline{e^{it'\beta \Delta}g}))dt'\right|\right|_{H^{s}}
\gtrsim  N^{s-1/2}.
\]
Therefore we obtain {\rm (\ref{flow_map_estimate})} because $s-1/2>2s-1$ for any $s<1/2$. 
\end{proof}
%
%
\appendix
\section{Bilinear estimates for $1$D Bourgain norm \label{appendix}}
In this section, we give the bilinear estimates for the standard $1$-dimensional Bourgain norm   
under the condition 
$(\alpha-\gamma)(\beta+\gamma)\neq 0$ and $\alpha \beta \gamma (1/\alpha-1/\beta-1/\gamma)\neq 0$. 
Which estimates imply the well-posedness of (\ref{NLS_sys}) for $1>s\geq 1/2$ as the solution $(u,v,w)$ be in the
Bourgain space 
$X^{s}_{\alpha}([0,T])\times X^{s}_{\beta}([0,T])\times X^{s}_{\gamma}([0,T])$.   
\begin{lemm}\label{modul_est_1d}
Let $\sigma_{1}$, $\sigma_{2}$, $\sigma_{3} \in \R \backslash \{0\}$ satisfy $(\sigma_{2}+\sigma_{3})(\sigma_{3}+\sigma_{1})\neq 0$ 
and $(\tau_{1},\xi_{1})$, $(\tau_{2}, \xi_{2})$, $(\tau_{3}, \xi_{3})\in \R\times \R$ satisfy $\tau_{1}+\tau_{2}+\tau_{3}=0$, $\xi_{1}+\xi_{2}+\xi_{3}=0$. 
If there exist $1\leq i,j\leq 3$ such that $|\xi_{i}|\ll |\xi_{j}|$, then we have
\[
\max_{1\leq j\leq 3}|\tau_{j}+\sigma_{j} \xi_{j}^{2}|
\gtrsim \xi_{3}^{2}. 
\]
\end{lemm}
\begin{proof}For the case $\sigma_{1}+\sigma_{2}\neq 0$, proof was complete in Lemma~\ref{modul_est}. 
We assume $\sigma_{1}+\sigma_{2}= 0$. Then we have
\[
\begin{split}
M_{0}:=&\max\{|\tau_{1}+\sigma_{1}\xi_{1}^{2}|, |\tau_{2}+\sigma_{2}\xi_{2}^{2}|, |\tau_{3}+\sigma_{3}\xi_{3}^{2}|\}\\
&\gtrsim |\sigma_{1} \xi_{1}^{2}+\sigma_{2} \xi_{2}^{2}+\sigma_{3} \xi_{3}^{2}|\\
&=|\xi_{3}||(\sigma_{1} +\sigma_{3} )\xi_{3}+2\sigma_{1} \xi_{2}|\\
&=|\xi_{3}||(\sigma_{2} +\sigma_{3} )\xi_{3}+2\sigma_{2} \xi_{1}|
\end{split}
\]
by the triangle inequality. Therefore if $|\xi_{i}|\ll |\xi_{j}|$ for some $1\leq i,j\leq 3$, then we have $M_{0}\gtrsim \xi_{3}^{2}$. 
\end{proof}
\begin{lemm}\label{integral_est}
We assume $\sigma_{1}$, $\sigma_{2}$, $\sigma_{3}\in \R\backslash \{0\}$ satisfy $\theta :=\sigma_{1} \sigma_{2} \sigma_{3} (1/\sigma_{1}+1/\sigma_{2}+1/\sigma_{3})\neq 0$. 
For any $(\tau, \xi) \in \R\times \R$ with $|\xi|\geq 1$ and $b>1/2$, we have
\begin{equation}\label{integral_estimate}
\int_{\R}\int_{\R}\frac{d\tau_{1}d\xi_{1}}{\langle \tau_{1}+\sigma_{1}\xi_{1}^{2}\rangle^{2b}\langle \tau-\tau_{1} +\sigma_{2} (\xi-\xi_{1})^{2}\rangle^{2b}}
\lesssim
\langle (\sigma_{1}+\sigma_{2})(\tau -\sigma_{3}\xi^{2})+\theta \xi^{2}\rangle^{-1/2}
\end{equation}
and
\begin{equation}\label{integral_estimate_2}
\int_{|\xi_{1}|\gg |\xi -\xi_{1}|\ {\rm or}\ |\xi_{1}|\ll |\xi -\xi_{1}|}\int_{\R}\frac{d\tau_{1}d\xi_{1}}{\langle \tau_{1}+\sigma_{1}\xi_{1}^{2}\rangle^{2b}\langle \tau-\tau_{1} +\sigma_{2} (\xi-\xi_{1})^{2}\rangle^{2b}}
\lesssim
\langle \xi\rangle^{-1},
\end{equation}
where implicit constants in $\ll$ actually depend on $\sigma_{1}$, $\sigma_{2}$. 
\end{lemm}
\begin{proof}
We put $I(\tau ,\xi ):=$(L.H.S of (\ref{integral_estimate})). 
By Lemma\ 2.3,\ (2.8) in \cite{KPV96}, we have
\[
I(\tau ,\xi )\lesssim \int_{\R}\frac{d\xi_{1}}{\langle \sigma_{1}\xi_{1}^{2}+\sigma_{2}(\xi -\xi_{1})^{2}+\sigma_{3}\xi^{2}+(\tau -\sigma_{3}\xi^{2}) \rangle^{2b}}. 
\]
We change the variable $\xi_{1}\mapsto \mu$ as $\mu =\sigma_{1}\xi_{1}^{2}+\sigma_{2}(\xi -\xi_{1})^{2}+\sigma_{3}\xi^{2}$, then we have  
\[
d\mu =2|\sigma_{1}\xi_{1}-\sigma_{2}(\xi -\xi_{1})|d\xi_{1}\sim \left|(\sigma_{1}+\sigma_{2})\mu -\theta \xi^{2}\right|^{1/2} d\xi_{1}. 
\]
Therefore if $\sigma_{1}+\sigma_{2}=0$, we obtain
\[
I(\tau ,\xi )\lesssim \frac{1}{|\xi |}\int_{\R}\frac{d\mu}{\langle \mu+(\tau -\sigma_{3}\xi^{2}) \rangle^{2b}}\lesssim \langle \xi \rangle^{-1} 
\]
for $b>1/2$ since $\theta \neq 0$ and $|\xi |\geq 1$. While if $\sigma_{1}+\sigma_{2}\neq 0$, we obtain
\[
I(\tau ,\xi )\lesssim \int_{\R}\frac{d\mu}{\langle \mu+(\tau -\sigma_{3}\xi^{2}) \rangle^{2b}\left| (\sigma_{1}+\sigma_{2})\mu -\theta \xi^{2}\right|^{1/2}}
\lesssim
\langle (\sigma_{1}+\sigma_{2})(\tau -\sigma_{3}\xi^{2})+\theta \xi^{2}\rangle^{-1/2}
\]
for $b>1/2$ by Lemma\ 2.3,\ (2.9) in \cite{KPV96}. The estimate (\ref{integral_estimate_2}) follows from
\[
d\mu =2|\sigma_{1}\xi_{1}-\sigma_{2}(\xi -\xi_{1})|d\xi_{1}\sim \max\{|\xi_{1}|,|\xi -\xi_{1}|\}d\xi_{1}\sim |\xi |d\xi_{1}
\]
when $|\xi_{1}|\gg |\xi -\xi_{1}|$ or $|\xi_{1}|\ll |\xi -\xi_{1}|$.
\end{proof}
\begin{prop}\label{Bourgain_be_prop}
We assume $d=1$ and $\sigma_{1}$, $\sigma_{2}$, $\sigma_{3}\in \R\backslash\{0\}$ satisfy 
$(\sigma_{1}+\sigma_{3})(\sigma_{2}+\sigma_{3})\neq 0$ and $\theta :=\sigma_{1} \sigma_{2} \sigma_{3} (1/\sigma_{1}+1/\sigma_{2}+1/\sigma_{3})\neq 0$. 
Then for $3/4\geq b>1/2$ and $1>s\geq 1/2$, we have 
\begin{align}
||(\partial_{x}u_{3})u_{2}||_{X^{s,b-1}_{-\sigma_{1}}}
&\lesssim ||u_{3}||_{X^{s,b}_{\sigma_{3}}}||u_{2}||_{X^{s,b}_{\sigma_{2}}},\label{Bourgain_be_1}\\
||\partial_{x}(u_{1}u_{2})||_{X^{s,b-1}_{-\sigma_{3}}}
&\lesssim ||u_{1}||_{X^{s,b}_{\sigma_{1}}}||u_{2}||_{X^{s,b}_{\sigma_{2}}},\label{Bourgain_be_2}
\end{align}
where
\[
||u||_{X^{s,b}_{\sigma}}:=||\langle \xi \rangle^{s}\langle \tau +\sigma \xi^{2}\rangle^{b}\widetilde{u}||_{L^{2}_{\tau \xi}}. 
\]
\end{prop}
\begin{proof}
We prove only (\ref{Bourgain_be_2}) since the proof of (\ref{Bourgain_be_1}) is similar. 
By the Cauchy-Schwarz inequality, we have
\[
||\partial_{x}(u_{1}u_{2})||_{X^{s,b-1}_{-\sigma_{3}}}
\lesssim ||I||_{L^{\infty}_{\tau \xi}}||u_{1}||_{X^{s,b}_{\sigma_{1}}}||u_{2}||_{X^{s,b}_{\sigma_{2}}},
\]
where
\[
I(\tau , \xi ):=\left(\frac{\langle \xi \rangle^{2s}|\xi|^{2}}{\langle \tau -\sigma_{3} \xi^{2}\rangle^{2(1-b)}}
\int_{\R}\int_{\R}\frac{\langle \xi_{1}\rangle^{-2s}\langle \xi -\xi_{1}\rangle^{-2s}}{\langle \tau_{1}+\sigma_{1}\xi_{1}^{2}\rangle^{2b}\langle \tau-\tau_{1} +\sigma_{2} (\xi-\xi_{1})^{2}\rangle^{2b}}d\tau_{1}d\xi_{1}\right)^{1/2}. 
\]
It is enough to prove $I(\tau,\xi)\lesssim 1$ for $|\xi|\geq 1$. 
For fixed $(\tau ,\xi )\in \R\times \R$, we divide $\R\times \R$ into three regions $S_{1}$, $S_{2}$, $S_{3}$ as
\[
\begin{split}
S_{1}&:=\{(\tau_{1},\xi_{1})\in \R\times \R|\ |\xi|\ll |\xi_{1}|\}\\
S_{2}&:=\{(\tau_{1},\xi_{1})\in \R\times \R|\ |\xi|\gtrsim |\xi_{1}|,\ 
\max\{|\tau_{1}+\sigma_{1}\xi_{1}^{2}|, |\tau-\tau_{1}+\sigma_{2}(\xi -\xi_{1})^{2}|\}\gtrsim \xi^{2}\}\\
S_{3}&:=\{(\tau_{1},\xi_{1})\in \R\times \R|\ |\xi|\gtrsim |\xi_{1}|,\ 
\max\{|\tau_{1}+\sigma_{1}\xi_{1}^{2}|, |\tau-\tau_{1}+\sigma_{2}(\xi -\xi_{1})^{2}|\}\ll \xi^{2}\}
\end{split}
\]

First, we consider the region $S_{1}$. 
For any $(\tau_{1},\xi_{1})\in S_{1}$, we have
\[
\langle \xi \rangle^{2s}|\xi|^{2}\langle \xi_{1}\rangle^{-2s}\langle \xi -\xi_{1}\rangle^{-2s}\lesssim \langle \xi \rangle^{2-2s}
\]
because $|\xi|\ll |\xi_{1}|\sim |\xi -\xi_{1}|$. 
Therefore, we have
\[
\begin{split}
I(\tau ,\xi )
&\lesssim \left(\frac{\langle \xi \rangle^{2-2s}}{\langle \tau -\sigma_{3} \xi^{2}\rangle^{2(1-b)}\langle (\sigma_{1}+\sigma_{2})(\tau -\sigma_{3}\xi^{2})+\theta\xi^{2}
\rangle^{1/2}}\right)^{1/2}
\end{split}
\]
for $b>1/2$ by (\ref{integral_estimate}). 
Because $\theta \neq 0$, 
\[
\xi ^{2}=\frac{1}{\theta}\left\{(\sigma_{1}+\sigma_{2})(\tau -\sigma_{3}\xi ^{2})+\theta \xi ^{2}-(\sigma_{1}+\sigma_{2})(\tau -\sigma_{3}\xi ^{2})\right\}. 
\]
Therefore we obtain 
\[
\begin{split}
I(\tau, \xi )&\lesssim \left(\frac{1}{\langle \tau -\sigma_{3}\xi ^{2}\rangle^{s-(2b-1)}\langle (\sigma_{1}+\sigma_{2})(\tau -\sigma_{3}\xi ^{2})+\theta \xi^{2}\rangle^{1/2}}\right.\\
&\ \ \ \ \ \ \ \ \ \ \ \ \ \ \ \ \ \ \ \ \left.+\frac{1}{\langle \tau -\sigma_{3} \xi ^{2}\rangle^{2(1-b)}\langle(\sigma_{1}+\sigma_{2})(\tau -\sigma_{3}\xi ^{2})+\theta \xi^{2}\rangle^{s-1/2}}\right)^{1/2}\\
&\lesssim 1
\end{split}
\]
for $3/4\geq b>1/2$ and $1>s\geq 1/2$. 

Second, we consider the region $S_{2}$. 
We assume $|\tau-\tau_{1}+\sigma_{2}(\xi -\xi_{1})^{2}|\gtrsim \xi^{2}$ $(\gtrsim |\xi|^{1/b}|\xi_{1}|^{1-1/2b}|\xi -\xi_{1}|^{1-1/2b})$ 
since for the case $|\tau_{1}+\sigma_{1}\xi_{1}^{2}|\gtrsim \xi^{2}$ is same argument. 
Then, we have
\[
I(\tau , \xi )\lesssim \left(\frac{\langle \xi \rangle^{2s}}{\langle \tau -\sigma_{3} \xi^{2}\rangle^{2(1-b)}}
\int_{\R}\int_{\R}\frac{\langle \xi_{1}\rangle^{1-2b-2s}\langle \xi -\xi_{1}\rangle^{1-2b-2s}}{\langle \tau_{1}+\sigma_{1}\xi_{1}^{2}\rangle^{2b}}d\tau_{1}d\xi_{1}\right)^{1/2}. 
\]
Because
\[
\int_{\R}\frac{d\tau_{1}}{\langle \tau_{1}+\sigma_{1}\xi_{1}^{2}\rangle^{2b}}\lesssim 1
\]
for $b>1/2$, we obtain
\[
\begin{split}
I(\tau , \xi )&\lesssim \left(\frac{\langle \xi \rangle^{2s}}{\langle \tau -\sigma_{3}\xi^{2}\rangle^{2(1-b)}}\int_{\R}\frac{d\xi_{1}}{\langle \xi_{1}\rangle^{2s+2b-1}\langle \xi -\xi_{1}\rangle^{2s+2b-1}}\right)^{1/2}\\
&\lesssim \left(\frac{1}{\langle \tau -\sigma_{3} \xi ^{2}\rangle^{2(1-b)}\langle \xi \rangle^{2b-1}}\right)^{1/2}\\
&\lesssim 1
\end{split}
\]
for $1\geq b>1/2$ and $s\geq 1/2$ by Lemma\ 2.3,\ (2.8) in \cite{KPV96}. 

Finally, we consider the region $S_{3}$.
To begin with, we consider the case $|\tau -\sigma_{3}\xi^{2}|\gtrsim \xi^{2}$. 
Then we have
\[
\frac{\langle \xi \rangle^{2s}|\xi |^{2}}{\langle \tau -\sigma_{3} \xi^{2}\rangle^{2(1-b)}}\langle \xi_{1}\rangle^{-2s}\langle \xi -\xi_{1}\rangle^{-2s}
\lesssim 
\begin{cases}
\langle \xi\rangle^{4b-2-2s}\ \ {\rm if}\ \ |\xi_{1}|\sim |\xi -\xi_{1}|\\
\langle \xi\rangle^{4b-2}\ \ {\rm if}\ \ |\xi_{1}|\gg |\xi -\xi_{1}|\ {\rm or}\ |\xi_{1}|\ll |\xi -\xi_{1}|
\end{cases}
\]
since $|\xi |\sim \max\{|\xi_{1}|,|\xi -\xi_{1}|\}$ for any $(\tau,\xi)\in S_{3}$. Therefore we obtain
\[
\begin{split}
I(\tau ,\xi )&\lesssim \left(\frac{\langle \xi\rangle^{4b-2-2s}}{\langle (\sigma_{1}+\sigma_{2})(\tau -\sigma_{3}\xi^{2})+\theta\xi^{2}
\rangle^{1/2}}+\langle \xi\rangle^{4b-3}\right)^{1/2}\lesssim 1
\end{split}
\]
for $3/4\geq b>1/2$ and $s\geq 1/2$ by (\ref{integral_estimate}) and (\ref{integral_estimate_2}). 
Next, we consider the case $|\tau -\sigma_{3}\xi^{2}|\ll \xi^{2}$. Because $(\sigma_{1}+\sigma_{3})(\sigma_{2}+\sigma_{3})\neq 0$ and
\[
\max\{|\tau_{1}+\sigma_{1}\xi_{1}^{2}|,|\tau -\tau_{1}+\sigma_{2}(\xi -\xi_{1})^{2}|,|\tau-\sigma_{3}\xi^{2}|\}\ll \xi^{2}, 
\]
we have $|\xi|\sim |\xi -\xi_{1}|\sim |\xi_{1}|$ by Lemma~\ref{modul_est_1d}. Therefore, we have
\[
I(\tau ,\xi )
\lesssim \left(\frac{\langle \xi \rangle^{2-2s}}{\langle \tau -\sigma_{3} \xi^{2}\rangle^{2(1-b)}\langle(\sigma_{1}+\sigma_{2})(\tau -\sigma_{3}\xi^{2})+\theta\xi^{2}
\rangle^{1/2}}\right)^{1/2}
\]
for $b>1/2$ by (\ref{integral_estimate}). 
Because $\theta \neq 0$ and $|\tau-\sigma_{3}\xi^{2}|\ll \xi^{2}$, we have
\[
I(\tau ,\xi )\lesssim \left(\frac{\langle \xi \rangle^{1-2s}}{\langle \tau -\sigma_{3} \xi^{2}\rangle^{2(1-b)}}\right)^{1/2}\lesssim 1
\]
for $1\geq b>1/2$ and $s\geq 1/2$. 
\end{proof}
\begin{cor}\label{Bourgain_be_cor}
We assume $d=1$ and $\alpha$, $\beta$, $\gamma\in \R\backslash\{0\}$ satisfy $(\alpha -\gamma )(\beta +\gamma)\neq 0$ and $\alpha \beta \gamma (1/\alpha-1/\beta-1/\gamma)\neq 0$. Then for $3/4\geq b>1/2$ and $1>s\geq 1/2$, we have 
\begin{align}
||(\partial_{x}w)v||_{X^{s,b-1}_{\alpha}}&\lesssim ||w||_{X^{s,b}_{\gamma}}||v||_{X^{s,b}_{\beta}},\label{Bourgain_be_al}\\
||(\partial_{x}\overline{w})u||_{X^{s,b-1}_{\beta}}&\lesssim ||w||_{X^{s,b}_{\gamma}}||u||_{X^{s,b}_{\alpha}},\label{Bourgain_be_be}\\
||\partial_{x}(u\overline{v})||_{X^{s,b-1}_{\gamma}}&\lesssim ||u||_{X^{s,b}_{\alpha}}||v||_{X^{s,b}_{\beta}},\label{Bourgain_be_ga}.
\end{align}
\end{cor}
\begin{proof}
(\ref{Bourgain_be_al}) follows from (\ref{Bourgain_be_1}) with $(u_{2},u_{3})=(v,w)$ and $(\sigma_{1},\sigma_{2},\sigma_{3})=(-\alpha,\beta,\gamma)$. 
(\ref{Bourgain_be_be}) follows from (\ref{Bourgain_be_1}) with $(u_{2},u_{3})=(u,\overline{w})$ and $(\sigma_{1},\sigma_{2},\sigma_{3})=(\alpha,-\beta,-\gamma)$. 
(\ref{Bourgain_be_ga}) follows from (\ref{Bourgain_be_2}) with $(u_{1},u_{2})=(u,\overline{v})$ and $(\sigma_{1},\sigma_{2},\sigma_{3})=(\alpha,-\beta,-\gamma)$.
\end{proof}
Theorem~\ref{wellposed_2}\ {\rm (iii)} under the condition $1>s\geq 1/2$, $\theta =\alpha \beta \gamma (1/\alpha-1/\beta-1/\gamma) <0$ and $(\alpha -\gamma )(\beta +\gamma )\neq 0$ follows from  Lemma\ 2.1 in \cite{GTV97} and Corollary~\ref{Bourgain_be_cor}. 
\section*{acknowledgements}
The author would like to express his appreciation to Kotaro Tsugawa for many discussions and very valuable comments. 

\end{document}